\documentclass[reqno,12pt]{amsart}
\usepackage[colorlinks=true, linkcolor=blue, citecolor=blue]{hyperref}

\usepackage{amssymb}
\usepackage{amsmath, graphicx, rotating}
\usepackage{color}
\usepackage{soul}
\usepackage[dvipsnames]{xcolor}

\allowdisplaybreaks
\usepackage{ifthen}
\usepackage{xkeyval}
\usepackage{todonotes}
\usepackage[T1]{fontenc}
\usepackage{lmodern}
\usepackage[english]{babel}

\usepackage{ upgreek }
\usepackage{stmaryrd}
\SetSymbolFont{stmry}{bold}{U}{stmry}{m}{n}
\usepackage{amsthm}
\usepackage{float}

\usepackage{ bbm }
\usepackage{ stmaryrd }
\usepackage{ mathrsfs }
\usepackage{ frcursive }
\usepackage{ comment }

\usepackage{pgf, tikz}
\usetikzlibrary{shapes}
\usepackage{varioref}
\usepackage{enumitem}

\definecolor{rouge}{rgb}{0.7,0.00,0.00}
\definecolor{vert}{rgb}{0.00,0.5,0.00}
\definecolor{bleu}{rgb}{0.00,0.00,0.8}
\usepackage[margin=1in]{geometry}
\newtheorem{theorem}{Theorem}[section]
\newtheorem*{theorem*}{Theorem}
\newtheorem{lemma}[theorem]{Lemma}

\newtheorem{proposition}[theorem]{Proposition}

\labelformat{hypothesis}{\textbf{M\kern-0.1mm#1}}

\newtheorem{condition}{Condition}

\newtheorem{conditionA}{A\kern-0.1mm}
\labelformat{conditionA}{\textbf{A\kern-0.1mm#1}}

\labelformat{conditionH}{\textbf{H\kern-0.1mm#1}}

\theoremstyle{definition}
\newtheorem{example}[theorem]{Example}
\newtheorem{remark}[theorem]{Remark}

\def \eref#1{\hbox{(\ref{#1})}}

\numberwithin{equation}{section}

\def\geq{\geqslant}
\def\leq{\leqslant}

\def\RR{\mathbb{R}}

\def\EE{\mathbb{E}}

\def \eref#1{\hbox{(\ref{#1})}}

\def\EE{\mathbb{ E}}

\begin{document}
	
	\title[Diffusion Approximation for SDEs with State-Dependent Switching]
	{Diffusion Approximation for Slow-Fast SDEs with State-Dependent Switching}

\author{Xiaobin Sun}
\curraddr[Sun, X.]{ School of Mathematics and Statistics/RIMS, Jiangsu Normal University, Xuzhou, 221116, People's Re
public of China}
\email{xbsun@jsnu.edu.cn}

\author{Jue Wang}
\curraddr[Wang, J.]{ School of Mathematics and Statistics, Jiangsu Normal University, Xuzhou, 221116, People's Re
public of China}
\email{2020221246@jsnu.edu.cn}

\author{Yingchao Xie}
\curraddr[Xie, Y.]{ School of Mathematics and Statistics/RIMS,  Jiangsu Normal University, Xuzhou, 221116, China}
\email{ycxie@jsnu.edu.cn}
	
	\begin{abstract}
		In this paper, we study the diffusion approximation for slow-fast stochastic differential equations with state-dependent switching, where the slow component $X^{\varepsilon}$ is the solution of a stochastic differential equation with  additional homogenization term, while the fast component $\alpha^{\varepsilon}$ is a switching process. We first prove the weak convergence of $\{X^\varepsilon\}_{0<\varepsilon\leq 1}$ to $\bar{X}$ in the space of continuous functions, as $\varepsilon\rightarrow 0$. Using the martingale problem approach and Poisson equation associated with a Markov chain, we identify this weak limiting process as the unique solution $\bar{X}$ of a new stochastic differential equation, which has new drift and diffusion terms that differ from those in the original equation. Next, we prove the order $1/2$ of weak convergence of $X^{\varepsilon}_t$ to $\bar{X}_t$ by applying suitable test functions $\phi$, for any $t\in [0, T]$. Additionally, we provide an example to illustrate that the order we achieve is optimal.
\end{abstract}

\maketitle

\section{Introduction}

Stochastic differential equations (SDEs) with state-dependent switching are employed to model the dynamics of a complex system characterized by a countable number of scenarios, denoted as $\mathbb{S}$. In each scenario, the system evolves a certain diffusion, and it can switch between scenarios and the switching depends on the current location or state of the system. Usually, it can be described as follows:
\begin{align*}\label{SDEsMS}
\begin{cases} dX_t=b(X_t, \alpha_t)dt+\sigma(X_t, \alpha_t)dW_t,\\
\mathbb{P}\{\alpha_{t+\Delta}=j|\alpha_{t}=i, X_s, \alpha_s, s\leq
t\}=\left\{\begin{array}{l}
\displaystyle q_{ij}(X_t)\Delta+o(\Delta),~~~~~~i\neq j,\\
1+q_{ii}(X_t)\Delta+o(\Delta),~~~i=j,\end{array}\right.
\end{cases}
\end{align*}
where $\{W_t\}_{t\geq 0}$ is a standard $d$-dimensional Brownian motion, $b$ and $\sigma$ two measurable functions from $\mathbb{R}^n\times\mathbb{S}$ to $\mathbb{R}^n$ and $\mathbb{R}^n\times \mathbb{R}^d$ respectively, and $Q(x)= (q_{ij}(x))_{i,j \in \mathbb{S}}$ is a $Q$-matrix depending on $x\in \mathbb{R}^n$. SDEs with state-dependent switching are widely used in various fields, including finance, engineering, biology, and physics. Therefore, it has attracted interest from researchers, who are actively working on this topic and investigate their potential applications. For more results in this area, we refer e.g. \cite{MY2006,SX2014,XZW2021,YZ2010} and references therein.

The diffusion approximation problem in slow-fast SDEs mainly concerns the weak convergence of singularly perturbed SDEs which involves a homogenization term. In the beginning of the 21st century, Pardoux and Veretennikov \cite{PV2001,PV2003,PV2005} use the technique of the Poisson equation to study the diffusion approximation problem for slow-fast SDEs, and taking the martingale problem approach to characterize the limiting process.
For more results on this subject, we refer e.g. \cite{BK2004,BS2023,FW2021,HLS2025,RX2021,WR2015}.

In the above framework, we present a small positive parameter $\varepsilon$ that indicates the ratio of the time scales for the diffusion process and the switching process, that is the diffusion process occurs on a slower time scale, while the switching process occurs at a much faster rate. Therefore, we consider the following slow-fast SDEs with state dependent switching.
\begin{align*}
\begin{cases} dX_t^{\varepsilon}=\frac{1}{\sqrt{\varepsilon}}K(X_t^{\varepsilon}, \alpha_t^{\varepsilon})dt+b(X_t^{\varepsilon}, \alpha_t^{\varepsilon})dt+\sigma(X_t^{\varepsilon}, \alpha_t^{\varepsilon})dW_t,\\
\mathbb{P}\{\alpha^{\varepsilon}_{t+\Delta}=j|\alpha_{t}^{\varepsilon}=i, X_s^{\varepsilon}, \alpha_s^{\varepsilon}, s\leq
t\}=\left\{\begin{array}{l}
\displaystyle Q^{\varepsilon}_{ij}(X_t^{\varepsilon})\Delta+o(\Delta),~~~~~~i\neq j,\\
1+Q^{\varepsilon}_{ii}(X_t^{\varepsilon})\Delta+o(\Delta),~~~i=j,\end{array}\right.
\end{cases}
\end{align*}
where $K/{\sqrt{\varepsilon}}$ represents the homogenization term, which is significant in the study of partial differential equations (see \cite{HP2008,HP2021}). Meanwhile, the switching process itself consists of two distinct time scales, represented by the state-dependent transition matrix $Q^{\varepsilon}(x)=\varepsilon^{-1}Q(x)+\varepsilon^{-1/2}\tilde{Q}(x)$. By the theory of diffusion approximation, the homogenization $K{\sqrt{\varepsilon}}$ and $\varepsilon^{-1/2}\tilde{Q}$  will contribute to an additional diffusion term and a drift term, respectively, in the limiting equation.

Without the homogenization in the previously mentioned model, the asymptotic behavior of the slow component $X^{\varepsilon}$ as $\varepsilon\to 0$ is referred to as the averaging principle. In the context of state-independent switching, where $Q^{\varepsilon}(x)$ is constant matrix, many scholars have studied the averaging principle (see e.g. \cite{BYY2017,CWW2024, LLSW2023, PXYZ2018, YY2004}). In the context of state-dependent switching case, Il'in et. al. \cite{IKY1999} study the weak convergence analysis of a class of singularly perturbed switching diffusions with fast and slow motions. Faggionato et. al. \cite{FGR2010} studied the averaging principle in terms of convergence in probability in the case  $\sigma\equiv 0$ and  $\mathbb{S}$ is finite, then Pakdaman et al. \cite{PTW2012} used an asymptotic expansion technique to establish a central limit theorem. Mao and Shao \cite{MS2022} studied the averaging principle for the case $\mathbb{S}$ is countably infinite. Subsequently, the averaging principle for a specific class of multi-scale piecewise-deterministic Markov processes in infinite dimensions were investigated in \cite{GT2012} and \cite{GT2014}.  In the case that the fast process has multiple invariant measures, Freidlin and Koralov \cite{FK2021},  along with Goddard et al. \cite{GOPS2023}, consider its averaging principle for the slow process to be described by SDEs and ODEs (ordinary differential equations), respectively.  However, the corresponding weak convergence order has not been addressed in these papers.

Recently, the authors \cite{SX2023+} have established theory of the Poisson equation related to a Markov chain, which is  applied to prove the averaging principles and central limit theorems.
Therefore, it make possible to further consider the diffusion approximation problem.
More precisely, we first aim to study the asymptotic behavior, as $\varepsilon\to 0$, of the slow component $X^{\varepsilon}$  converges weakly to $\bar X$ in the continuous function space $C([0,T],\mathbb{R}^n)$. Then by martingale problem approach, the weak limiting process can be precisely characterized as the unique solution of a new SDE:
\begin{equation*}
d\bar{X}_t = \bar B(\bar{X}_t) \, dt + \Theta(\bar{X}_t) \, d\bar{W}_t,
\end{equation*}
where $\bar{W}$ is an $n$-dimensional standard Brownian motion on the probability space $(\Omega, \mathcal{F}, \mathbb{P})$, and the drift and diffusion coefficients $\bar B$ and $\Theta$ are expressed in terms of the solution of Poisson equation.

It is also worthy to point the challenges through a comparison with the classical slow-fast diffusion process.  For instance,  an important step involves studying the tightness of $\{X^{\varepsilon}\}_{\varepsilon\in(0,1]}$ in $C([0,T];\RR^{n})$.  In the context of fast process is a diffusion process, scholars would like use the tight criterion in $C([0,T];\RR^{n})$. However, note that the switching process is a jump process, so we need to establish the tightness of $\{X^{\varepsilon}\}_{\varepsilon\in(0,1]}$ in the $c\`adl\`ag$ function space $\mathbb{D}([0,T];\mathbb{R}^{n})$.

It is widely recognized that finding optimal/explicit convergence rates is crucial in multi-scale numerical problem, for instance, Br\'ehier \cite{B2020} mentioned that "\emph{the analysis of the full error of the scheme requires as a preliminary step to estimate the error in the averaging principle}". Therefore, the second aim of this paper is to further study the weak convergence order of $X^{\varepsilon}_t$ to $\bar{X}_t$. More precisely, we indent to get a weak convergence with order $1/2$ in $C^4_p(\RR^n)$, which is the space of functions with polynomial growth of the fourth continuous partial derivative, that is, for $\phi\in C^4_p(\RR^n)$,
\begin{align*}
\sup_{0\leq t\leq T}|\EE \phi(X^{\varepsilon}_t)-\EE \phi(\bar{X}_t)|\leq C\left(1+|x|^{k}\right)\varepsilon^{\frac{1}{2}}, \quad \forall \varepsilon\in(0,1].
\end{align*}
$1/2$ will shown to be optimal by a concrete example. The main technique is based on a combination of the  Poisson equation and Kolmogorov equation, which have been confirmed as a powerful technique that has successfully been used to obtain the optimal strong or weak convergence rates in various slow-fast stochastic systems, see e.g.  \cite{BS2023, B2020,PS2008,RSX2021,SSWX2024,SXX2022,SX2023} and references therein.

The remainder construction of the paper as follows. In section 2, we give some notations and assumptions, as well as the detailed presentation of our main results. In section 3, we give some a prior estimates of the solution and study the well-posedness of the averaged equation. We give the detailed proof of our first and second main results in the Section 4 and Section 5, respectively. Throughout this paper, the symbols $k$, $k_p$, $C_T$, and $C_{p,T}$ denote positive constants which may change from line to line, where the subscript parameters are used to emphasize that the constants depends on these parameters.

\section{Main results}

In this section, we will give some notations and assumptions firstly, then the main results are stated.

\subsection{Notations and assumptions}
Denote $\mathbb{S}:=\{1,2,\ldots,m_0\}$ for some $m_0\leq\infty$. Let $\mathscr{B}_b(\mathbb{S})$ be the set of bounded function on $\mathbb{S}$, and denote $\|f\|_{\infty}:=\sup_{i\in \mathbb{S}}|f(i)|$. Denote $\|\mu-\nu\|_{\rm{var}}$ as the total variation distance between probability measures $\mu$ and $\nu$ on $\mathbb{S}$.
We will use $|\cdot|$ and  $\langle\cdot, \cdot\rangle$ to denote the Euclidean  norm and the inner product, respectively, and denote $\|\cdot\|$ as the matrix norm or the operator norm if without confusion. For any positive integers $k, l$, we will use $\RR^k\otimes\RR^l$ to denote the set of $k\times l$ matrices. If $m_0=\infty$,
$\RR^\infty\otimes \RR^\infty$ stands for the set of matrices $M=(M_{ij})_{i,j\in \mathbb{S}}$ with  $\|M\|_{\ell}:=\sup_{i\in\mathbb{S}}\sum_{j\in \mathbb{S}}|m_{ij}|<\infty$.

\vspace{0.1cm}
For positive integers $l_1,l_2,k$, the following function spaces will be used in this paper:

(1) $C^k(\RR^n,\RR^{l_1}\otimes\RR^{l_2})$ (resp. $C^k_b(\RR^n,\RR^{l_1}\otimes\RR^{l_2})$ or $C^k_p(\RR^n,\RR^{l_1}\otimes\RR^{l_2})$) is the space of $\RR^{l_1}\otimes\RR^{l_2}$-valued functions $\varphi(x)$ on $\RR^n$ with all partial derivatives of up to order $k$ being continuous (resp. continuous and bounded or continuous and polynomial growth). $C^k(\RR^n,\RR^{\infty}\otimes\RR^{\infty})$ (resp. $C^k_b(\RR^n, \RR^{\infty}\otimes\RR^{\infty})$ or $C^k_p(\RR^n, \RR^{\infty}\otimes\RR^{\infty})$) stands for space of $\RR^{\infty}\otimes\RR^{\infty}$-valued functions $\varphi(x)$ on $\RR^n$ with  all partial derivatives of up to order $k$ being continuous (resp. continuous and bounded or continuous and polynomial growth).

(2) $C^k(\RR^n\times \mathbb{S},\RR^{l_1}\otimes \RR^{l_2})$ (resp. $C^k_{b}(\RR^n\times \mathbb{S},\RR^{l_1}\otimes \RR^{l_2})$ or $C^k_{p}(\RR^n\times \mathbb{S},\RR^{l_1}\otimes \RR^{l_2})$) is the space of all mappings $\varphi(x,i): \RR^n\times \mathbb{S}\rightarrow \RR^{l_1}\otimes \RR^{l_2}$ with $\varphi(\cdot,i)\in C^k(\RR^n,\RR^{l_1}\otimes \RR^{l_2})$ (resp. $C^k_{b}(\RR^n,\RR^{l_1}\otimes \RR^{l_2})$ or $C^k_p(\RR^n,\RR^{l_1}\otimes \RR^{l_2})$) for any $i\in\mathbb{S}$, and $\|\partial^j_x\varphi(x,\cdot)\|_{\infty}$ are finite (resp. bounded or polynomial growth) with respective to $x$, for any $j=0,\ldots,k$.

\vspace{0.1cm}
In this paper, we focus on the following slow-fast SDEs with state-dependent switching:
\begin{align}\label{Mainequation}
\begin{cases} dX_t^{\varepsilon}=\frac{1}{\sqrt{\varepsilon}}K(X_t^{\varepsilon}, \alpha_t^{\varepsilon})dt+b(X_t^{\varepsilon}, \alpha_t^{\varepsilon})dt+\sigma(X_t^{\varepsilon}, \alpha_t^{\varepsilon})dW_t,\\
\mathbb{P}\{\alpha^{\varepsilon}_{t+\Delta}=j|\alpha_{t}^{\varepsilon}=i, X_s^{\varepsilon}, \alpha_s^{\varepsilon}, s\leq
t\}=\left\{\begin{array}{l}
\displaystyle Q^{\varepsilon}_{ij}(X_t^{\varepsilon})\Delta+o(\Delta),~~~~~~i\neq j,\\
1+Q^{\varepsilon}_{ii}(X_t^{\varepsilon})\Delta+o(\Delta),~~~i=j,\end{array}\right.
\end{cases}
\end{align}
with initial value $(X_0^{\varepsilon},\alpha_0^{\varepsilon})=(x,\alpha)\in \mathbb{R}^n\times \mathbb{S}$, where $\{W_t\}_{t\geq 0}$ is a standard $d$-dimensional standard Brownian motion on a complete filtered probability space
$(\Omega,\mathscr{F},\{\mathscr{F}_{t}\}_{t\geq0},\mathbb{P})$, $Q^{\varepsilon}(x):=\varepsilon^{-1}Q(x)+\varepsilon^{-1/2}\tilde{Q}(x)$, and
$$
K: \mathbb{R}^n\times\mathbb{S}\rightarrow \mathbb{R}^n; \quad b: \mathbb{R}^n\times\mathbb{S}\rightarrow \mathbb{R}^n; \quad \sigma:\mathbb{R}^n\times\mathbb{S}\rightarrow \mathbb{R}^n\times \mathbb{R}^d;
$$
$$
Q:\mathbb{R}^n\rightarrow \mathbb{R}^{m_0}\times \mathbb{R}^{m_0};\quad  \tilde{Q}:\mathbb{R}^n\rightarrow \mathbb{R}^{m_0}\times \mathbb{R}^{m_0}.
$$

Suppose that $K$, $b$, $\sigma$, $Q$ and $\tilde{Q}$ satisfy the following conditions:
\begin{conditionA}\label{A1} Assume there exists $C > 0$ such that for any $x, y \in \mathbb{R}^n$,
\begin{align}\label{ConA11}\begin{split}
&\|b(x,\cdot) - b(y,\cdot)\|_{\infty} \leq C|x - y|,\quad \|b(x,\cdot)\|_{\infty} \leq C(1 + |x|),\\
&\|\sigma(x,\cdot) - \sigma(y,\cdot)\|_{\infty} \leq C|x - y|,\quad \sup_{x \in \mathbb{R}^n}\|\sigma(x,\cdot)\|_{\infty} \leq C.
\end{split}
\end{align}
Furthermore, $\sigma$ satisfies the following uniform non-degeneracy condition:
\begin{equation}
\inf_{x \in \mathbb{R}^n, i \in \mathbb{S}, z \in \mathbb{R}^n \backslash \{0\}} \frac{\langle (\sigma(x,i)\sigma^{*}(x,i)) \cdot z, z \rangle}{|z|^2} > 0. \label{NonD}
\end{equation}
\end{conditionA}

\begin{conditionA}\label{A2}
(i) $Q(x)$ and $\tilde{Q}(x)$ are measurable functions from $\mathbb{R}^n \rightarrow \mathbb{R}^{m_0} \otimes \mathbb{R}^{m_0}$ and are both conservative, i.e., for all $x \in \mathbb{R}^n$,
\[
q_{ij}(x) \geq 0 \quad \text{for all} \quad i \neq j \in \mathbb{S}, \quad
\sum_{j \in \mathbb{S}} q_{ij}(x) = 0 \quad \text{for all} \quad i \in \mathbb{S},
\]
\[
\tilde{q}_{ij}(x) \geq 0 \quad \text{for all} \quad i \neq j \in \mathbb{S}, \quad
\sum_{j \in \mathbb{S}} \tilde{q}_{ij}(x) = 0 \quad \text{for all} \quad i \in \mathbb{S}.
\]

(ii) $Q(x)$ is irreducible, i.e., for all $x \in \mathbb{R}^n$, the equation
\[
\mu^x Q(x) = 0, \quad \text{with} \quad \sum_{i \in \mathbb{S}} \mu^x_i = 1,
\]
has a unique solution $\mu^x = (\mu^x_1, \mu^x_2, \ldots, \mu^x_{m_0})$ with $\mu^x_i > 0$ for all $i \in \mathbb{S}$. Furthermore, let $P^x_t := (p^x_{ij}(t))_{i,j \in \mathbb{S}}$ be the transition probability matrix of $Q(x)$. Assume that $P^x_t$ satisfies uniform ergodicity, i.e., there exist $C > 0, \lambda > 0$ such that
\begin{equation}
\sup_{i \in \mathbb{S}, x \in \mathbb{R}^n} \|p^x_{i \cdot}(t) - \mu^x\|_{\text{var}} \leq Ce^{-\lambda t}, \quad \forall t > 0. \label{ergodicity}
\end{equation}

(iii) Assume $Q \in C^2(\mathbb{R}^n, \mathbb{R}^{m_0} \otimes \mathbb{R}^{m_0})$ and $\tilde{Q} \in C^1(\mathbb{R}^n, \mathbb{R}^{m_0} \otimes \mathbb{R}^{m_0})$ and there exists $C > 0$ such that
\begin{align}
\sup_{x \in \mathbb{R}^n} \left( \|\nabla Q(x)\|_{\ell} + \|\nabla^2 Q(x)\|_{\ell} + \|\nabla \tilde{Q}(x)\|_{\ell} \right) \leq C. \label{LipQ}
\end{align}
\begin{align}\label{FinteA}
\begin{split}
&A(x) := \sum_{i \in \mathbb{S}} \sum_{j \in \mathbb{S} \backslash \{i\}} q_{ij}(x) \leq C(1 + |x|), \quad \forall x \in \mathbb{R}^n, \\
&\tilde{A}(x) := \sum_{i \in \mathbb{S}} \sum_{j \in \mathbb{S} \backslash \{i\}} \tilde{q}_{ij}(x) \leq C(1 + |x|), \quad \forall x \in \mathbb{R}^n.
\end{split}
\end{align}
\end{conditionA}

\begin{conditionA}\label{A3}
Assume that $K \in C^2_b(\mathbb{R}^n \times \mathbb{S}, \mathbb{R}^n)$ and satisfies the following "centering condition",
\begin{equation}
\sum_{i \in \mathbb{S}} K(x, i) \mu^x_i = 0, \quad \forall x \in \mathbb{R}^n. \label{CenCon}
\end{equation}
\end{conditionA}

\begin{remark} In the following, we give some comments on the assumptions.

(i) Conditions \eref{ConA11} and \eref{FinteA} imply Eq.\eref{Mainequation} has a unique strong solution by \cite[Theorem 3.3]{NY2016}.

(ii) Assumption \ref{A2} ensures that $Q(x)$ generates a unique Markov process $\{\alpha^{x}_t\}_{t\geq 0}$. The condition \eref{ergodicity} is called strong ergodicity for $Q(x)$ uniformly in $x$. It is worthy to mention the book \cite{C2005} and references therein, which provide various sufficient conditions for strong ergodicity \eref{ergodicity} of continuous-time Markov chains, such the birth-death process.

(iii) When $m_0$ is finite,  \eref{LipQ} imply \eref{FinteA} holds. However, for $m_0=\infty$,  \eref{FinteA} is necessary, since this and finite moment of the solution $X^{\varepsilon}$ of any order are use to control the switching term.
\end{remark}

\subsection{Main results}
This subsection presents the main results of this paper. To this end, we introduce the following Poisson equation
\begin{equation}\label{PE1}
-Q(x)\Phi(x,\cdot)(i) = K(x,i)
\end{equation}
which is equivalent to
\begin{equation*}
-Q(x)\Phi^l(x,\cdot)(i) = K^l(x,i), \quad l = 1,2,\ldots,n.
\end{equation*}
Here, $\Phi(x,i) = (\Phi^1(x,i), \ldots, \Phi^n(x,i))$.

For  \eref{PE1}, we can obtain the following conclusion, whose proof can be founded in  \cite[Theorem 2.2]{SX2023+}.
\begin{proposition}\label{Poisson}
Suppose that assumptions \ref{A2} and \ref{A3} hold. Let
\begin{equation}
\Phi(x,i) := \int_{0}^{\infty} \mathbb{E} K(x,\alpha_{t}^{x,i}) \, dt. \label{SPE}
\end{equation}
where $\alpha_{t}^{x,i}$ is the unique $\mathbb{S}$-valued Markov chain generated by generator $Q(x)$ with initial
value $\alpha_{0}^{x,i}=i\in\mathbb{S}$. Then $\Phi(x,i)$ is a solution of Eq.(\ref{PE1}), and there exists a constant $C > 0$ such that for any $x \in \mathbb{R}^n$, we have
\begin{equation}\label{EPS}\begin{split}
&\|\Phi(x,\cdot)\|_{\infty} \leq C \|K(x,\cdot)\|_{\infty};\\
&\|\partial_{x} \Phi(x,\cdot)\|_{\infty} \leq C \left[ \|K(x,\cdot)\|_{\infty} \|\nabla Q(x)\|_{\ell} + \|\partial_{x}K(x,\cdot)\|_{\infty} \right];\\
&\|\partial_{x}^2 \Phi(x,\cdot)\|_{\infty} \leq C \left[ \|K(x,\cdot)\|_{\infty} \left( \|\nabla Q(x)\|_{\ell} + \|\nabla Q(x)\|_{\ell}^2 + \|\nabla^2 Q(x)\|_{\ell} \right) \right.\\
&\quad\quad\quad\quad\quad\quad\quad\quad\quad\left. + \|\partial_{x}K(x,\cdot)\|_{\infty} \|\nabla Q(x)\|_{\ell} + \|\partial_{x}^2 K(x,\cdot)\|_{\infty} \right].
\end{split}
\end{equation}
\end{proposition}

\vspace{0.2cm}
Subsequently, we introduce  notation:
\begin{equation}\label{DN}\begin{split}
&(K \otimes \Phi)(x,j) := ((K \otimes \Phi)_{ml}(x,j))_{\{1 \leq m,l \leq n\}} := \left( K_m(x,j) \Phi^l(x,j) \right)_{\{1 \leq m, l \leq n\}}; \\
&\partial_{x} \Phi_K(x,j) := (\partial_{x} \Phi_K^l(x,j))_{\{1 \leq l \leq n\}} := \left( \langle \partial_{x} \Phi^l(x,j), K(x,j) \rangle \right)_{\{1 \leq l \leq n\}}; \\
&F(x,j) := (F^l(x,j))_{1 \leq l \leq n} := \left( \tilde{Q}(x) \Phi^l(x,\cdot)(j) \right)_{1 \leq l \leq n} = \left( \sum_{i \in \mathbb{S}} \tilde{Q}_{ji}(x) \Phi^l(x,i) \right)_{1 \leq l \leq n};\\
&B(x,j):= (b + \partial_{x} \Phi_K + F)(x,j);\\
&\Sigma(x,j) := \left[ \sigma \sigma^* + (K \otimes \Phi) + (K \otimes \Phi)^* \right](x,j).
\end{split}\end{equation}
And for any matrix-valued function $\phi(x,\alpha)$ on $\mathbb{R}^n \times \mathbb{S}$, we define
\[
\bar{\phi}(x) = \sum_{j \in \mathbb{S}} \phi(x,j) \mu_j^x.
\]

Now, we present our first main result, that is a weak convergence result of $X^{\varepsilon}$ in the space of continuous functions $C([0,T],\mathbb{R}^n)$.

\begin{theorem}\label{main result 1}
Suppose that assumptions \ref{A1}-\ref{A3} hold. Then for any initial value $(x,\alpha) \in \mathbb{R}^n\times \mathbb{S}$ and $T > 0$, $\{X^{\varepsilon}\}_{\varepsilon > 0}$ converges weakly in $C([0,T];\mathbb{R}^n)$ to the unique solution $\bar{X}$ of the following equation
\begin{equation}
d\bar{X}_t = \bar{B}(\bar{X}_t) \, dt + \bar{\Sigma}^{1/2}(\bar{X}_t) \, d\bar{W}_t\label{bar}
\end{equation}
as $\varepsilon \to 0$,  where $\bar{W}$ is a standard $n$-dimensional standard Brownian motion on the probability space $(\Omega, \mathcal{F}, \mathbb{P})$,  $\bar{\Sigma}^{1/2}(x)$ is the square root of $\bar{\Sigma}(x)$.
\end{theorem}

\begin{remark}
The well-posedness of Eq.\eref{bar} will be proved in Lemma \ref{PMbarX2} below. Roughly speaking, the homogenization $K/\!{\sqrt{\varepsilon}}$ will contribute to an additional diffusion $\overline{K\! \otimes\! \Phi\!+\! (K\! \otimes\! \Phi)^*} $ and drift $\overline{\partial_x \Phi_{K}}$.  Furthermore,  the combination of the homogenization with $\varepsilon^{-1/2}\tilde{Q}$ within $Q^{\varepsilon}$ will contribute to an additional drift $\bar{F}$. Nevertheless, if we consider $Q^{\varepsilon}(x)=\varepsilon^{-1}Q(x)+\varepsilon^{-\beta}\tilde{Q}(x)$ for some $\beta \in [0,1/2)$, the additional drift $\bar{F}$ in the limiting Eq.\eref{bar} vanishes.
\end{remark}

Next, we will further investigate the convergence rate of the weak convergence of $X^{\varepsilon}_t$ to $\bar{X}_t$ for $t \le T$, i.e., for appropriate test functions $\phi$, we will study the order of convergence for $\left|\mathbb{E}\phi(X^{\varepsilon}_t) - \mathbb{E}\phi(\bar{X}_t)\right|$. To obtain the optimal weak convergence order, we require an additional condition as follows:

\begin{conditionA}\label{A4} Assume $K \in C^4_b(\mathbb{R}^n \times \mathbb{S}, \mathbb{R}^{n})$, $\sigma \in C^4_b(\mathbb{R}^n \times \mathbb{S}, \mathbb{R}^{n} \times \mathbb{R}^d)$, $b \in C^4(\mathbb{R}^n \times \mathbb{S}, \mathbb{R}^{n})$, and $Q, \tilde{Q} \in C^4(\mathbb{R}^n, \mathbb{R}^{m_0} \otimes \mathbb{R}^{m_0})$. Moreover,
\begin{align*}
\sup_{x \in \mathbb{R}^n} \sum_{j=1}^4 \left[ \|\partial^j_{x}{b(x,i)}\|_{\infty} + \|\nabla^j{Q(x)}\|_{\ell} +\left \|\nabla^j{\tilde{Q}(x)}\right\|_{\ell} \right] \leq C.
\end{align*}
\end{conditionA}

\begin{remark}
Using the same argument as in \cite[Remark 4.6]{SX2023+},  the condition \ref{A4} and uniform non-degeneracy condition \eref{NonD} ensure $\bar B \in C^4(\mathbb{R}^n \times \mathbb{S}, \mathbb{R}^{n})$,  $\bar\Sigma^{1/2} \in C^4(\mathbb{R}^n \times \mathbb{S}, \mathbb{R}^{n} \otimes \mathbb{R}^n)$. Moreover,
\begin{align}
\sup_{x \in \mathbb{R}^n} \sum_{j=1}^4 \left[ \left\|\nabla^j{\bar B(x)}\right\| + \left\|\nabla^j{\bar\Sigma^{1/2}(x)}\right\| \right] \leq C.\label{RE2.4}
\end{align}
\end{remark}

\vspace{0.1cm}
Below, we present the second main result of this paper:
\begin{theorem}\label{main result 2}
    Suppose that conditions \ref{A1}-\ref{A4} hold. Then, for initial values $(x,\alpha) \in \mathbb{R}^n \times \mathbb{S}$, $T > 0$, and $\phi \in C^4_p(\mathbb{R}^n)$, there exists $k>0$ such that
    \begin{eqnarray}
    \sup_{0\leq t\leq T}\left|\mathbb{E} \phi(X^{\varepsilon}_t) - \mathbb{E} \phi(\bar{X}_t)\right| \leq C_{T,\phi}\left(1 + |x|^{k}\right)\varepsilon^{\frac{1}{2}}, \quad \forall \varepsilon \in (0,1],\label{AVWO}
    \end{eqnarray}
    where $C_{T,\phi}$ depending on $T$ and $\phi$, and $\bar{X}$ is the solution to Equation \eref{bar}.
\end{theorem}

\begin{example} We illustrate the optimality of the convergence order $1/2$ through the following example. For simplicity, we only consider 1-D case..
\begin{eqnarray*}
dX^{\varepsilon}_t = \frac{1}{\sqrt{\varepsilon}}K(\alpha^{\varepsilon}_t)dt + dW_t,\quad X^{\varepsilon}_0 = x \in \mathbb{R},
\end{eqnarray*}
where $\alpha^{\varepsilon}$ is a Markov chain taking values in $\mathbb{S}=\{1,2\}$ with $\alpha^{\varepsilon}_0 = 1$ and generator
\[
Q^{\varepsilon} := \varepsilon^{-1}Q := \varepsilon^{-1}\left(
\begin{array}{cc}
-1 & 1 \\
1 & -1 \\
\end{array}
\right).
\]
It is easy to verify that $Q$ has a unique invariant probability measure $\mu = \left(1/2,1/2\right)$. The transition probability $P^\varepsilon_t$ of $\alpha^{\varepsilon}$ has the form:
\[
P^{\varepsilon}_t = \frac{1}{2}\left(
\begin{array}{cc}
1 + e^{-2t/\varepsilon} & 1 - e^{-2t/\varepsilon} \\
1 - e^{-2t/\varepsilon} & 1 + e^{-2t/\varepsilon} \\
\end{array}
\right).
\]
Put $K(1) = 1$ and $K(2) = -1$. It is easy to see that $K$ satisfies the "centering condition," and the Poisson equation
$$-Q\Phi(i) = K(i)$$
has a solution
$$\Phi(1) = \int^{\infty}_0 \mathbb{E}K(\alpha^{1}_{t})dt = 1/2,\quad \Phi(2) = \int^{\infty}_0 \mathbb{E}K(\alpha^{2}_{t})dt = -1/2,$$
where $\{\alpha^{i}_{t}\}_{t\geq 0}$ is a Markov chain with the generator $Q$ and initial value $\alpha^{i}_{0}=i$, $i=1,2$. Thus, the corresponding averaged equation is
\[
\bar{X}_t = x + \sqrt{2}\bar{W}_t,
\]
where $\bar{W}$ is a 1-D standard Brownian motion. Taking $\phi(x) = x$, we have
\begin{align*}
\sup_{t \in [0,T]}|\mathbb{E}\phi(X^{\varepsilon}_t) - \mathbb{E}\phi(\bar{X}_t)| &= \sup_{t \in [0,T]}|\mathbb{E}X^{\varepsilon}_t - \mathbb{E}\bar{X}_t| = \sup_{t \in [0,T]}\left|\frac{1}{\sqrt{\varepsilon}}\int^t_0 \mathbb{E}K(\alpha^{\varepsilon}_s) \, ds\right|\\
&= \sup_{t \in [0,T]}\left|\frac{1}{\sqrt{\varepsilon}}\int^t_0 \left[\mathbb{P}(\alpha^{\varepsilon}_s = 1) - \mathbb{P}(\alpha^{\varepsilon}_s = 2)\right] \, ds\right|\\
&= \sup_{t \in [0,T]} \left|\frac{1}{\sqrt{\varepsilon}}\int^t_0 \left[p^{\varepsilon}_{11}(s) - p^{\varepsilon}_{12}(s)\right] \, ds\right|\\
&= \frac{1}{\sqrt{\varepsilon}}\int^T_0 e^{-2s/\varepsilon} \, ds = \frac{\sqrt{\varepsilon}}{2}\left(1 - e^{-\frac{2T}{\varepsilon}}\right) = O(\sqrt{\varepsilon}),
\end{align*}
which shows that the weak convergence order $1/2$ in \eref{AVWO} is optimal.
\end{example}

\section{A Priori estimates and well-posedness of Eq.\eref{bar}}
In this section, we establish some preliminary results on moment estimates for the solution $(X_t^{\varepsilon}, \alpha_t^{\varepsilon})$ of the stochastic system \eref{Mainequation}, then we study the well-posedness of Eq.\eref{bar}.

\begin{lemma} \label{PMbarX2}
Suppose the conditions \ref{A1}-\ref{A3} hold. For $T > 0$ and $p \geq 2$, there exists $C_{p,T} > 0$ such that
\begin{equation}\label{5}
\mathbb{E}\left(\sup_{0\leq t\leq T}|X^\varepsilon_t|^p\right)\leq C_{p,T}\left(1+|x|^p\right).
\end{equation}
\end{lemma}

\begin{proof}
The proof is divided into two steps.

\textbf{Step 1}: We are going to prove that, for $T > 0$ and $p \geq 2$, there exists $C_{p,T} > 0$ such that for $\varepsilon \in (0,1]$,
\begin{eqnarray}
\mathbb{E}\left[\sup_{0 \leq t \leq T}\left|\frac{1}{\sqrt{\varepsilon}}\int_0^t K(X_{s}^{\varepsilon},\alpha^{\varepsilon}_{s})ds\right|^p\right] \leq C_{p,T}\left(1 + |x|^p + \int_0^T \mathbb{E}|X^{\varepsilon}_t|^p dt\right).  \label{F3.1.1}
\end{eqnarray}

In fact, according to \cite{BBG1999}, $\alpha^{\varepsilon}$ can be expressed as the following
\begin{equation*} \label{2.3}
\alpha^{\varepsilon}_t = \alpha + \int_0^t \int_{[0, \infty)} g^{\varepsilon}(X^{\varepsilon}_{s-},\alpha^{\varepsilon}_{s-}, z)N(ds, dz),
\end{equation*}
where $N(dt, dz)$ is a Poisson random measure defined on $\Omega \times \mathscr{B}(\mathbb{R}_{+}) \times \mathscr{B}(\mathbb{R}_{+})$ with L\'evy measure being the Lebesgue measure,  $N(dt, dz)$ is independent of $W$. The function $g^{\varepsilon}(x,i, z)$ is defined as
\[
g^{\varepsilon}(x,i, z) = \sum_{j \in \mathbb{S} \backslash \{i\}} (j - i)1_{z \in \triangle^{\varepsilon}_{ij}(x)}, \quad i \in \mathbb{S},
\]
where $\triangle^{\varepsilon}_{ij}(x)$ are a sequence of consecutive (left-closed, right-open) intervals on $[0,\infty)$ arranged in lexicographical order on $\mathbb{S} \times \mathbb{S}$, each of length $\varepsilon^{-1}q_{ij}(x) + \varepsilon^{-1/2}\tilde{q}_{ij}(x)$, with $\Delta^{\varepsilon}_{12} = [0,\varepsilon^{-1}q_{12}(x) + \varepsilon^{-1/2}\tilde{q}_{12}(x))$.

For a solution $\Phi$ of Eq.\eref{PE1},  \ref{A2},  \ref{A3} and Proposition \ref{Poisson} imply $\Phi \in C^2_b(\mathbb{R}^n \times \mathbb{S}, \mathbb{R}^n)$, that is,
\begin{align}
\|\Phi(x,\cdot)\|_{\infty} + \|\partial_x\Phi(x,\cdot)\|_{\infty} + \|\partial_x^2\Phi(x,\cdot)\|_{\infty} \leq C. \label{Phi}
\end{align}
Ito's formula yields that
\begin{equation}
\begin{aligned}\label{F3.15}
\Phi(X_t^{\varepsilon},\alpha_t^{\varepsilon}) = &\Phi(x,\alpha) + \frac{1}{\sqrt{\varepsilon}}\int_0^t \partial_x\Phi_K(X_s^{\varepsilon},\alpha_s^{\varepsilon}) \, ds \\
&+ \int_0^t \partial_x\Phi(X_s^{\varepsilon},\alpha_s^{\varepsilon}) \cdot b(X_s^{\varepsilon},\alpha_s^{\varepsilon}) \, ds + \int_0^t \partial_x\Phi(X_s^{\varepsilon},\alpha_s^{\varepsilon}) \cdot \sigma(X_s^{\varepsilon},\alpha_s^{\varepsilon}) \, dW_s \\
&+ \frac{1}{2}\int_0^t \text{Tr}\Big[\partial_x^2\Phi(X_s^{\varepsilon},\alpha_s^{\varepsilon}) \cdot \sigma\sigma^\ast(X_s^{\varepsilon},\alpha_s^{\varepsilon})\Big] \, ds \\
&+ \int_0^t \int_{[0,\infty)} \left[\Phi(X_{s-}^{\varepsilon},\alpha_{s-}^{\varepsilon} + g^{\varepsilon}(X_{s-}^{\varepsilon},\alpha_{s-}^{\varepsilon},z)) - \Phi(X_{s-}^{\varepsilon},\alpha_{s-}^{\varepsilon})\right] \tilde{N}(ds,dz) \\
&+ \frac{1}{\varepsilon}\int_0^t Q(X_s^{\varepsilon})\Phi(X_s^{\varepsilon},\cdot)(\alpha_s^{\varepsilon}) \, ds + \frac{1}{\sqrt{\varepsilon}}\int_0^t F(X_s^{\varepsilon},\alpha_s^{\varepsilon}) \, ds.
\end{aligned}
\end{equation}
Thus, for any $p \geq 2$,
\begin{equation}
\begin{aligned}
&\mathbb{E}\left[\sup_{0 \leq t \leq T}\left|\int_0^t K(X_s^{\varepsilon},\alpha_s^{\varepsilon}) \, ds\right|^p\right] \\
\leq &C_p\varepsilon^p\mathbb{E}\left[\sup_{0 \leq t \leq T}\left|\Phi(x,\alpha) - \Phi(X_t^{\varepsilon},\alpha_t^{\varepsilon})\right|^p + \sup_{0 \leq t \leq T}\left|\int_0^t \partial_x\Phi(X_s^{\varepsilon},\alpha_s^{\varepsilon})\sigma(X_s^{\varepsilon},\alpha_s^{\varepsilon}) \, dW_s\right|^p\right. \\
&\quad\quad\left. + \int_0^T |\partial_x\Phi(X_s^{\varepsilon},\alpha_s^{\varepsilon})b(X_s^{\varepsilon},\alpha_s^{\varepsilon})|^p \, ds + \int_0^T \left|\text{Tr}\left[\partial_x^2\Phi(X_s^{\varepsilon},\alpha_s^{\varepsilon}) \cdot \sigma\sigma^\ast(X_s^{\varepsilon},\alpha_s^{\varepsilon})\right]\right|^p \, ds\right] \\
&+ C_p\varepsilon^p \mathbb{E} \left[\sup_{0 \leq t \leq T} \left| \int_0^t \int_{[0,\infty)} \left[\Phi(X_s^{\varepsilon},\alpha_s^{\varepsilon} + g^{\varepsilon}(X_s^{\varepsilon},\alpha_s^{\varepsilon},z)) - \Phi(X_s^{\varepsilon},\alpha_s^{\varepsilon})\right] \tilde{N}(ds,dz) \right|^p\right] \\
&+ C_p\varepsilon^\frac{p}{2}\mathbb{E}\left[\sup_{0 \leq t \leq T}\left(\left|\int_0^t \partial_x\Phi(X_s^{\varepsilon},\alpha_s^{\varepsilon})K(X_s^{\varepsilon},\alpha_s^{\varepsilon}) \, ds\right|^p + \left|\int_0^t \tilde{Q}(X_s^{\varepsilon})\Phi(X_s^{\varepsilon},\cdot)(\alpha_s^{\varepsilon}) \, ds\right|^p\right)\right] \\
=: &\sum_{i=1}^3 J_i^{\varepsilon}(T). \label{barX2}
\end{aligned}
\end{equation}

By \eref{Phi}, we have
\begin{align}
J_1^{\varepsilon}(T) \leq C_p\varepsilon^p\left(1 + \int_0^T \mathbb{E}|X_t^{\varepsilon}|^p \, dt\right) . \label{J1}
\end{align}

For $J_2^{\varepsilon}$, using Kunita's first inequality (see \cite[Theorem 4.4.23]{A2009}),
\begin{equation}
\begin{aligned}\label{J2}
J_2^{\varepsilon}(T) \leq &C_p\varepsilon^p \mathbb{E} \left[ \sup_{0 \leq t \leq T} \left| \int_0^t \int_{[0,\infty)} \left[\Phi(X_s^{\varepsilon},\alpha_s^{\varepsilon} + g^{\varepsilon}(X_s^{\varepsilon},\alpha_s^{\varepsilon},z)) - \Phi(X_s^{\varepsilon},\alpha_s^{\varepsilon})\right] \tilde{N}(ds,dz) \right|^p \right] \\
\leq &\varepsilon^p C_p \mathbb{E} \left[ \int_0^T \int_{[0,\infty)} |\Phi(X_s^{\varepsilon},\alpha_s^{\varepsilon} + g^{\varepsilon}(X_s^{\varepsilon},\alpha_s^{\varepsilon},z)) - \Phi(X_s^{\varepsilon},\alpha_s^{\varepsilon})|^2 \, dz \, ds \right]^\frac{p}{2} \\
&+ \varepsilon^p C_p \mathbb{E} \int_0^T \int_{[0,\infty)} |\Phi(X_s^{\varepsilon},\alpha_s^{\varepsilon} + g^{\varepsilon}(X_s^{\varepsilon},\alpha_s^{\varepsilon},z)) - \Phi(X_s^{\varepsilon},\alpha_s^{\varepsilon})|^p \, dz \, ds \\
\leq &\varepsilon^p C_p \mathbb{E} \left[ \int_0^T \int_{[0,A(X_s^\varepsilon)\varepsilon^{-1} + B(X_s^\varepsilon)\varepsilon^{-1/2}]} |\Phi(X_s^{\varepsilon},\alpha_s^{\varepsilon} + g^{\varepsilon}(X_s^{\varepsilon},\alpha_s^{\varepsilon},z)) - \Phi(X_s^{\varepsilon},\alpha_s^{\varepsilon})|^2 \, dz \, ds \right]^\frac{p}{2} \\
&+ \varepsilon^p C_p \mathbb{E} \int_0^T \int_{[0,A(X_s^\varepsilon)\varepsilon^{-1} + B(X_s^\varepsilon)\varepsilon^{-1/2}]} |\Phi(X_s^{\varepsilon},\alpha_s^{\varepsilon} + g^{\varepsilon}(X_s^{\varepsilon},\alpha_s^{\varepsilon},z)) - \Phi(X_s^{\varepsilon},\alpha_s^{\varepsilon})|^p \, dz \, ds \\
\leq &\varepsilon^p C_p \mathbb{E} \left[ \int_0^T\!\! \left[A(X_s^\varepsilon)\varepsilon^{-1} \!+ B(X_s^\varepsilon)\varepsilon^{-1/2}\right] \, ds \right]^\frac{p}{2}\!+ \varepsilon^p C_p \mathbb{E} \int_0^T \!\!\left[A(X_s^\varepsilon)\varepsilon^{-1} \!+ B(X_s^\varepsilon)\varepsilon^{-1/2}\right] \, ds \\
\leq &\varepsilon^\frac{p}{2} C_{p,T}\left(1 + \int_0^T \mathbb{E}|X_s^\varepsilon|^p \, ds\right).
\end{aligned}
\end{equation}

Using \ref{A2}, \ref{A3} and \eref{Phi}, we can easily obtain
\begin{align}
J_3^{\varepsilon}(T)\leq C_p\varepsilon^{p/2}+C_p\varepsilon^{p/2}\int^T_0 \mathbb{E}\|\tilde{Q}(X^{\varepsilon}_t)\|^p_{\ell} \, dt
\leq C_p\varepsilon^{p/2}\left(1+\int^T_0 \mathbb{E}|X^{\varepsilon}_t|^p \, dt\right).\label{J3}
\end{align}

Therefore,  \eref{J1}-\eref{J3} yield that \eref{F3.1.1} holds.

\textbf{Step 2}: For  $p \geq 2$ and $s \in [0,T]$, using  $C_r$ inequality, we have
\begin{align}
|X_s^\varepsilon|^p \leq &C_p|x|^p + \frac{C_p}{\varepsilon^{p/2}}\left|\int_0^s K(X_{r}^{\varepsilon},\alpha^{\varepsilon}_{r})dr\right|^p + C_p\left|\int_0^s b(X_{r}^{\varepsilon},\alpha^{\varepsilon}_{r}) dr\right|^p \nonumber\\
&+ C_p\left|\int_0^s \sigma(X_{r}^{\varepsilon},\alpha^{\varepsilon}_{r})dW_r\right|^p. \label{ROS}
\end{align}
Using \eref{ROS}, \eref{F3.1.1},  Young's inequality and BDG inequality, we can obtain that, for $t \in [0,T]$,
\begin{align*}
\mathbb{E}\left(\sup_{0\leq s\leq t}|X_s^\varepsilon|^p\right)\leq C_{p,T}(1+|x|^p) + C_{p,T}\mathbb{E}\int_0^t |X_{s}^{\varepsilon}|^p \, ds
+ C_{p,T}\mathbb{E}\int_0^t |X_{s}^{\varepsilon}|^p \, ds.
\end{align*}
Gronwall's inequality yields
\begin{align*}
\mathbb{E}\left(\sup_{0\leq t\leq T}|X^\varepsilon_t|^p\right)\leq C_{p,T}\left(1+|x|^p\right).
\end{align*}
The proof is complete.
\end{proof}

We will study the well-posedness of Eq.\eref{bar} in the followings.

\begin{lemma} \label{PMbarX2}
Suppose the conditions \ref{A1}-\ref{A3} hold, then Eq.\eref{bar} has a unique strong solution $\bar{X}$, and for any $T > 0$ and $p > 0$, there exists $C_{p,T} > 0$ such that
\begin{equation}
\mathbb{E}\left(\sup_{0 \leq t \leq T}| \bar{X}_t|^{p}\right)\le C_{p,T}(1+|x|^p).  \label{barX2}
\end{equation}
\end{lemma}

\begin{proof}
We only need to prove
\begin{align}\label{Lipbarb}\begin{split}
&|\bar B(x)-\bar B(y)|\leq C(1+|x|+|y|)|x-y|,\quad |\bar B(x)|\leq C(1+|x|),\\
&\|\bar{\Sigma}^{1/2}(x)-\bar{\Sigma}^{1/2}(y)\|\leq C|x-y|.
\end{split}
\end{align}
It is easy to conclude that Eq. \eref{bar} has a unique strong solution and tthat estimate \eref{barX2} holds.

Indeed, from the definition of $F$,  \ref{A2}, \eref{LipQ} and  \eref{Phi} imply that, for any $x,y\in\RR^n$,
\begin{align*}
&\left\|F(x,\cdot)-F(y,\cdot)\right\|_\infty\leq C(1+|x|+|y|)|x-y|,\quad \left\|F(x,\cdot)\right\|_\infty\leq C(1+|x|).\\
&\left\|\partial_{x}\Phi_K(x,\cdot)-\partial_{x}\Phi_K(y,\cdot)\right\|_\infty\leq C|x-y|,\quad \left\|\partial_{x}\Phi_K(x,\cdot)\right\|_\infty\leq C.
\end{align*}
Thus it follows
\begin{align*}
\left\|B(x,\cdot)-B(y,\cdot)\right\|_\infty\leq C(1+|x|+|y|)|x-y|,\quad \left\|B(x,\cdot)\right\|_\infty\leq C(1+|x|).
\end{align*}
Therefore,
\begin{align}
|\bar{B}(x_1)-\bar{B}(x_2)|=&\left|\sum_{j\in \mathbb{S}}(B(x_1,j)\mu^{x_1}_j)-\sum_{j\in \mathbb{S}}(B(x_2,j)\mu^{x_1}_j)\right|\nonumber\\
\leq&\left|\sum_{j\in \mathbb{S}}(B(x_1,j)\mu^{x_1}_j)-\mathbb{E}B(x_1,\alpha^{x_1,i}_t)\right|+\left|\sum_{j\in \mathbb{S}}(B(x_2,j)\mu^{x_2}_j)-\mathbb{E}B(x_2,\alpha^{x_2,i}_t)\right|\nonumber\\
&+\left|\mathbb{E}B(x_1,\alpha^{x_1,i}_t)-\mathbb{E}B(x_2,\alpha^{x_2,i}_t)\right|\nonumber\\
\leq&\|B(x_1,\cdot)\|_{\infty}\|p^{x_1}_{i\cdot}-\mu^{x_1}\|_{\text{var}}|x_1-x_2|
+\|B(x_2,\cdot)\|_{\infty}\|p^{x_2}_{i\cdot}-\mu^{x_2}\|_{\text{var}}|x_1-x_2|\nonumber\\
&+\left|\mathbb{E}B(x_1,\alpha^{x_1,i}_t)-\mathbb{E}B(x_2,\alpha^{x_1,i}_t)\right|
+\left|\mathbb{E}B(x_2,\alpha^{x_1,i}_t)-\mathbb{E}B(x_2,\alpha^{x_2,i}_t)\right|\nonumber\\
\leq&C(1+|x_1|+|x_2|)e^{-\lambda t}|x_1-x_2|+C(1+|x_1|+|x_2|)|x_1-x_2|\nonumber\\
&+\left|\mathbb{E}B(x_2,\alpha^{x_1,i}_t)-\mathbb{E}B(x_2,\alpha^{x_2,i}_t)\right|.\label{F3.12}
\end{align}
\cite[Proposition 3.1]{SX2023+} implies that $P^x_t$ is differentiable with respect to $x$, and
\begin{align*}
\sup_{t \geq 0, i \in \mathbb{S}}\|\partial_{x}P^{x}_tf(i)\|\leq C\|\partial_{x}Q(x)\|_{\ell},
\end{align*}
which implies
\begin{align}
\left|\mathbb{E}B(x_2,\alpha^{x_1,i}_t)-\mathbb{E}B(x_2,\alpha^{x_2,i}_t)\right|
=&\left|P^{x_1}_t B(x_2,\cdot)(i)-P^{x_2}_t B(x_2,\cdot)(i)\right|\nonumber\\
\leq&C\|B(x_2,\cdot)\|_{\infty}|x_1-x_2|\leq C(1+|x_2|)|x_1-x_2|.\label{F3.13}
\end{align}
Thus using \eref{F3.13} and letting $t\rightarrow \infty$ in \eref{F3.12}, it is easy to see the first estimate of \eref{Lipbarb} holds.

Next, we prove that the second estimate of \eref{Lipbarb} holds. Note that by \ref{A1}, \ref{A3} and  \eref{Phi}, we have
$$
\|\Sigma(x,\cdot)-\Sigma(y,\cdot)\|_\infty\leq C|x-y|, \quad \sup_{x\in\mathbb{R}^n}\|\Sigma(x,\cdot)\|_\infty\leq C.
$$
Thus, by the similar argument above, we can obtain
$$
\left\|\bar{\Sigma}(x)-\bar{\Sigma}(y)\right\|\leq C|x-y|.
$$

From \eref{NonD}, we know that $\overline{\sigma\sigma^*}$ is uniformly non-degenerate, i.e.,
\begin{equation*}
\inf_{x\in\mathbb{R}^n,z\in\mathbb{R}^n\backslash\{0\}}\frac{\langle\overline{\sigma\sigma^*}(x)\cdot z,z\rangle}{|z|^2}>0.
\end{equation*}
Therefore, for any $x\in\mathbb{R}^n$, $\overline{\sigma\sigma^*}(x)$ is a symmetric positive-definite matrix. According to \cite[Lemma 4.2.4]{PTW2012}, the matrix $\overline{K\otimes\Phi+(K\otimes\Phi)^*}(x)$ is a symmetric non-negative definite matrix. Thus, $\bar{\Sigma}(x)$ is a symmetric positive-definite matrix and satisfies
\begin{equation*}
\inf_{x\in\mathbb{R}^n,z\in\mathbb{R}^n\backslash\{0\}}\frac{\langle\bar{\Sigma}(x)\cdot z,z\rangle}{|z|^2}>0.
\end{equation*}
Therefore, all eigenvalues of $\bar{\Sigma}(x)$ are strictly greater than some constant $\gamma > 0$ that does not depend on $x$. Then, according to the following fact (see \cite{PW2006}): let $A$ and $B$ be two $n \times n$ symmetric positive definite matrices with all eigenvalues greater than or equal to $\gamma > 0$, then it holds that
\[
\|A^{1/2} - B^{1/2}\| \leq \frac{1}{2\gamma}\|A - B\|.
\]
Utilizing the aforementioned fact, we can conclude that $\bar{\Sigma}^{1/2}(x)$ is Lipschitz,  i.e.,
\[
\left\|\bar{\Sigma}^{1/2}(x) - \bar{\Sigma}^{1/2}(y)\right\| \leq C|x - y|.
\]
This completes the proof.
\end{proof}

\section{Weak Convergence of $X^{\varepsilon}$ in $C([0,T],\mathbb{R}^n)$}

In this section, we will prove the weak convergence of $X^{\varepsilon}$ in $C([0,T],\mathbb{R}^n)$.  We first study the tightness of $\{X^{\varepsilon}\}$ in $C([0,T];\mathbb{R}^n)$,  then we give the proof of Theorem 2.3.

\subsection{Tightness of $\{X^{\varepsilon}\}$ in $C([0,T];\mathbb{R}^n)$}

\begin{proposition}\label{pro4.6}
If conditions A1-A3 hold, then for $T > 0$, the distribution of $\{X^\varepsilon\}_{\varepsilon\in(0,1]}$ is tight in $C([0,T];\mathbb{R}^n)$ .
\end{proposition}

\begin{proof}
For technical reasons, we first prove that $\{X^{\varepsilon}\}_{\varepsilon\le 1}$ is tight in $\mathbb{D}([0,T];\mathbb{R}^n)$.
Then utilizing the continuity of the paths of $X^{\varepsilon}$, we can conclude that $\{X^{\varepsilon}\}_{\varepsilon<1}$ is also tight in $C([0,T];\mathbb{R}^n)$.

According to the tightness criterion in the space $\mathbb{D}([0,T];\mathbb{R}^n)$ (see \cite[Chapter 6, Theorem 4.5]{JS2013}), it is sufficient to prove that $\{X^\varepsilon\}_{\varepsilon<1}$ satisfies the following conditions:

(i) For $\delta > 0$, there exist $\varepsilon_0 > 0$ and $\lambda > 0$ such that
\begin{equation}\label{t1}
\sup_{0<\varepsilon\leq \varepsilon_0}\mathbb{P}\left(\sup_{0\leq t\leq T}|X^\varepsilon_t|\geq \lambda\right)\leq \delta.
\end{equation}

(ii) For $\delta > 0$ and $K > 0$, there exist $\theta > 0$ and $\varepsilon_0 > 0$ such that
\begin{equation}\label{t2}
\sup_{0<\varepsilon\leq \varepsilon_0}\sup_{S_1,S_2\in \mathcal{T}^{T} , S_1\leq S_2\leq S_1+\theta}\mathbb{P}\left(|X^\varepsilon_{S_1}-X^\varepsilon_{S_2}|\geq K\right)\leq \delta,
\end{equation}
where $\mathcal{T}^{T}=\{S: S \text{ is an } \mathscr{F}_t\text{-stopping time such that } S \le T\}$. Then $\{X^{\varepsilon}\}_{\varepsilon\le 1}$ is tight in $\mathbb{D}([0,T];\mathbb{R}^n)$.

By \eref{barX2} and Chebyshev's inequality,  it is easy to prove that $\{X^\varepsilon\}$ satisfies  \eref{t1}. Next, we will verify \eref{t2}.  According to \eref{F3.15}, for any $S_1,S_2\in \mathcal{T}^{T}$ with $S_1\leq S_2\leq S_1+\theta$, we have
\begin{align*}
&\frac{1}{\sqrt{\varepsilon}}\int^{S_2}_{S_1} K(X_s^{\varepsilon},\alpha_s^{\varepsilon})ds=\frac{1}{\sqrt{\varepsilon}}\int^{S_2}_{S_1} -Q(X_s^{\varepsilon})\Phi(X_s^{\varepsilon},\cdot)(\alpha_s^{\varepsilon})ds\\
=&\sqrt{\varepsilon}\left[\Phi(X_{S_1}^{\varepsilon},\alpha_{S_1}^{\varepsilon})-\Phi(X_{S_2}^{\varepsilon},\alpha_{S_2}^{\varepsilon})+\int^{S_2}_{S_1} \partial_x\Phi(X_s^{\varepsilon},\alpha_s^{\varepsilon})b(X_s^{\varepsilon},\alpha_s^{\varepsilon})ds\right.\\
&+\left.\int^{S_2}_{S_1} \partial_x\Phi(X_s^{\varepsilon},\alpha_s^{\varepsilon})\sigma(X_s^{\varepsilon},\alpha_s^{\varepsilon})dW_s+\frac{1}{2}\int^{S_2}_{S_1} \text{Tr}\left[\partial_x^2\Phi(X_s^{\varepsilon},\alpha_s^{\varepsilon})\cdot\sigma\sigma^\ast(X_s^{\varepsilon},\alpha_s^{\varepsilon})\right]ds\right]\\
&+\sqrt{\varepsilon}\int^{S_2}_{S_1}\int_{[0,\infty)}\left[\Phi(X^{\varepsilon}_{s-},\alpha^{\varepsilon}_{s-}+g^{\varepsilon}(X^{\varepsilon}_{s-},\alpha^{\varepsilon}_{s-},z))-\Phi(X^{\varepsilon}_{s-},\alpha^{\varepsilon}_{s-})\right] \tilde{N}(ds,dz)\\
&+\int^{S_2}_{S_1} \partial_x\Phi_{K}(X_s^{\varepsilon},\alpha_s^{\varepsilon})ds+\int^{S_2}_{S_1} F(X_s^{\varepsilon},\alpha_s^{\varepsilon})ds.
\end{align*}
This, \eref{barX2} and \eref{Phi} yield that
\begin{align*}
&\mathbb{E}|X_{S_1}^{\varepsilon}-X_{S_2}^{\varepsilon}|^{2}\\
\leq&\frac{C}{\varepsilon}\mathbb{E}\left|\int^{S_2}_{S_1} K(X_s^{\varepsilon},\alpha_s^{\varepsilon})ds\right|^2+ C\mathbb{E}\left|\int_{S_1}^{S_2}b(X_{s}^{\varepsilon},\alpha^{\varepsilon}_{s})ds\right|^2+C\mathbb{E}\left|\int_{S_1}^{S_2}\sigma(X_{s}^{\varepsilon},\alpha^{\varepsilon}_{s})dW_s\right|^2\\
\leq& C_{T}\varepsilon (1+|x|^2)+C\mathbb{E}\left|\int^{S_2}_{S_1} \partial_x\Phi_K(X_s^{\varepsilon},\alpha_s^{\varepsilon}) ds\right|^2+\mathbb{E}\left|\int^{S_2}_{S_1} F(X_s^{\varepsilon},\alpha_s^{\varepsilon})ds\right|^2\\
&+C\varepsilon\mathbb{E}\left|\int^{S_2}_{S_1}\int_{[0,\infty)}\left[\Phi(X^{\varepsilon}_{s-},\alpha^{\varepsilon}_{s-}+g^{\varepsilon}(X^{\varepsilon}_{s-},\alpha^{\varepsilon}_{s-},z))-\Phi(X^{\varepsilon}_{s-},\alpha^{\varepsilon}_{s-})\right] \tilde{N}(ds,dz)\right|^2\\
&+C\mathbb{E}\left|\int_{S_1}^{S_2}b(X_{s}^{\varepsilon},\alpha^{\varepsilon}_{s})ds\right|^2+C\mathbb{E}\left|\int_{S_1}^{S_2}\sigma(X_{s}^{\varepsilon},\alpha^{\varepsilon}_{s})dW_s\right|^2\\
\leq&C_{T}\varepsilon (1+|x|^2)+C_{T}(\theta^2+\theta)\left(1+|x|^2\right).
\end{align*}
Chebyshev's inequality implies that $\{X^{\varepsilon}\}$ satisfies \eref{t2}. This completes the proof.
\end{proof}

\subsection{Proof of Theorem 2.3}

We first need the follow result.

\begin{proposition}\label{Pro4.2}
Suppose condition \ref{A2} holds. Let  $H(x,\alpha):\mathbb{R}^n\times\mathbb{S}\to\RR^n$  satisfy the "centering condition" \eref{CenCon} and
\[
\|H(x,\cdot)-H(y,\cdot)\|_{\infty}\leq C(1+|x|^k+|y|^k)|x-y|,\quad \|H(x,\cdot)\|_{\infty}\leq C(1+|x|^k)
\]
for some $k>0$. Then for $0\leq t_0\leq t\leq T$, we have
\begin{eqnarray}
\mathbb{E}\left|\int_{t_0}^tH(X_{s}^{\varepsilon},\alpha^{\varepsilon}_{s})ds\right|\leq C_T(1+|x|^{k+1})\sqrt{\varepsilon}.\label{ConAV}
\end{eqnarray}
\end{proposition}
\begin{proof}
Consider the Poisson equation:
\begin{eqnarray}
-Q(x)\hat{\Phi}(x,\cdot)(i) = H(x,i), \label{PEQ2}
\end{eqnarray}
By Proposition \ref{Poisson}, we know that \eref{PEQ2} has a solution $\hat\Phi$. Furthermore, we have
\begin{eqnarray*}
\left\|\hat\Phi(x,\cdot) - \hat\Phi(y,\cdot)\right\|_{\infty} \leq C\left(1 + |x|^k + |y|^k\right)|x - y|, \quad \left\|\hat\Phi(x,\cdot)\right\|_{\infty} \leq C\left(1 + |x|^k\right).
\end{eqnarray*}
Using mollification techniques, we can find a sequence of smooth functions $\{\hat{\Phi}_m(x,i)\}_{m\geq 1}$ with respect to $x$ such that
\begin{align}
\left\|\hat{\Phi}(x,\cdot) - \hat{\Phi}_m(x,\cdot)\right\|_\infty \leq C\left(1 + |x|^k\right) m^{-1}, \quad \left\|\hat{\Phi}_m(x,\cdot)\right\|_{\infty} \leq C\left(1 + |x|^k\right),\label{PhiApp1}
\end{align}
\begin{align}
\left\| \partial_x \hat{\Phi}_m(x,\cdot)\right\|_{\infty} \leq C\left(1 + |x|^k\right), \quad \left\| \partial^2_x \hat{\Phi}_m(x,\cdot)\right\|_{\infty} \leq C\left(1 + |x|^k\right) m.\label{PhiApp2}
\end{align}
Thus by estimate \eref{PhiApp1}, we obtain
\begin{align}
\mathbb{E}\left|\int_{t_0}^t H(X_{s}^{\varepsilon},\alpha^{\varepsilon}_{s})ds\right|^2= &\mathbb{E}\left|\int_{t_0}^t Q(X_{s}^{\varepsilon})\hat\Phi(X_{s}^{\varepsilon},\cdot)(\alpha^{\varepsilon}_{s})ds\right|^2\nonumber\\
\leq &2\mathbb{E}\left|\int_{t_0}^t \left[Q(X_{s}^{\varepsilon})\hat\Phi(X_{s}^{\varepsilon},\cdot)(\alpha^{\varepsilon}_{s}) - Q(X_{s}^{\varepsilon})\hat{\Phi}_m(X_{s}^{\varepsilon},\cdot)(\alpha^{\varepsilon}_{s})\right]ds\right|^2\nonumber\\
&+2\mathbb{E}\left|\int_{t_0}^t Q(X_{s}^{\varepsilon})\hat{\Phi}_m(X_{s}^{\varepsilon},\cdot)(\alpha^{\varepsilon}_{s})ds\right|^2\nonumber\\
\leq &C_{T}\mathbb{E}\int_{t_0}^t \|Q(X_{s}^{\varepsilon})\|^2_{\ell} \|\hat{\Phi}(X_{s}^{\varepsilon},\cdot) - \hat{\Phi}_m(X_{s}^{\varepsilon},\cdot)\|^2_{\infty} ds\nonumber\\
&+C\mathbb{E}\left|\int_{t_0}^t Q(X_{s}^{\varepsilon})\hat{\Phi}_m(X_{s}^{\varepsilon},\cdot)(\alpha^{\varepsilon}_{s})ds\right|^2\nonumber\\
\leq &C_{T}(1 + |x|^{2k+2})m^{-2} + C\mathbb{E}\left|\int_{t_0}^t Q(X_{s}^{\varepsilon})\hat{\Phi}_m(X_{s}^{\varepsilon},\cdot)(\alpha^{\varepsilon}_{s})ds\right|^2.\label{F4.19}
\end{align}
Applying It\^o's formula to the function $\hat\Phi_m$, one has
\begin{align*}
\hat{\Phi}_m(X_t^{\varepsilon},\alpha_t^{\varepsilon}) = &\hat{\Phi}_m(X_{t_0}^{\varepsilon},\alpha_{t_0}^{\varepsilon}) + \frac{1}{\sqrt{\varepsilon}}\int_{t_0}^t \partial_x\hat{\Phi}_m(X_s^{\varepsilon},\alpha_s^{\varepsilon}) \cdot K(X_s^{\varepsilon},\alpha_s^{\varepsilon})ds\\
&+ \int_{t_0}^t \partial_x\hat{\Phi}_m(X_s^{\varepsilon},\alpha_s^{\varepsilon})\cdot b(X_s^{\varepsilon},\alpha_s^{\varepsilon})ds + \int_{t_0}^t \partial_x\hat{\Phi}_m(X_s^{\varepsilon},\alpha_s^{\varepsilon})\cdot \sigma(X_s^{\varepsilon},\alpha_s^{\varepsilon})dW_s\\
&+ \frac{1}{2}\int_{t_0}^t \text{Tr}\Big[\partial_x^2\hat{\Phi}_m(X_s^{\varepsilon},\alpha_s^{\varepsilon})\cdot\sigma\sigma^\ast(X_s^{\varepsilon},\alpha_s^{\varepsilon})\Big]ds\\
&+ \int_{t_0}^t\int_{[0,\infty)}\left[\hat{\Phi}_m(X^{\varepsilon}_{s-},\alpha^{\varepsilon}_{s-}+g^{\varepsilon}(X^{\varepsilon}_{s-},\alpha^{\varepsilon}_{s-},z)) - \hat{\Phi}_m(X^{\varepsilon}_{s-},\alpha^{\varepsilon}_{s-})\right] \tilde{N}(ds,dz)\\
&+ \frac{1}{\varepsilon}\int_{t_0}^t Q(X_s^{\varepsilon})\hat{\Phi}_m(X_s^{\varepsilon},\cdot)(\alpha_s^{\varepsilon})ds + \frac{1}{\sqrt{\varepsilon}}\int^{t}_{t_0} \tilde Q(X_s^{\varepsilon})\hat{\Phi}_m(X_s^{\varepsilon},\cdot)(\alpha_s^{\varepsilon})ds.
\end{align*}
Therefore, we have
\begin{align*}
&-\int_{t_0}^t Q(X_{s}^{\varepsilon})\hat{\Phi}_m(X_{s}^{\varepsilon},\cdot)(\alpha_s^{\varepsilon})ds\nonumber\\
= &\varepsilon\Big\{\hat{\Phi}_m(X_{t_0}^{\varepsilon},\alpha_{t_0}^{\varepsilon}) - \hat{\Phi}_m(X_t^{\varepsilon},\alpha_t^{\varepsilon}) + \int_{t_0}^t \partial_x\hat{\Phi}_m(X_s^{\varepsilon},\alpha_s^{\varepsilon})\cdot b(X_s^{\varepsilon},\alpha_s^{\varepsilon})ds\\
&+ \int_{t_0}^t \partial_x\hat{\Phi}_m(X_s^{\varepsilon},\alpha_s^{\varepsilon})\cdot \sigma(X_s^{\varepsilon},\alpha_s^{\varepsilon})dW_s + \frac{1}{2}\int_{t_0}^t \text{Tr}\Big[\partial_x^2\hat{\Phi}_m(X_s^{\varepsilon},\alpha_s^{\varepsilon})\cdot\left(\sigma\sigma^\ast\right)(X_s^{\varepsilon},\alpha_s^{\varepsilon})\Big]ds\Big\}\\
&+ \sqrt{\varepsilon}\left[\int_{t_0}^t \partial_x\hat{\Phi}_m(X_s^{\varepsilon},\alpha_s^{\varepsilon}) \cdot K(X_s^{\varepsilon},\alpha_s^{\varepsilon})ds + \int_{t_0}^t \tilde Q(X_s^{\varepsilon})\hat{\Phi}_m(X_s^{\varepsilon},\cdot)(\alpha_s^{\varepsilon})ds\right]\\
&+ \varepsilon\int_{t_0}^t\int_{[0,\infty)}\left[ \hat{\Phi}_m(X_{s-}^{\varepsilon},\alpha^{\varepsilon}_{s-}+g^{\varepsilon}(X_{s-}^{\varepsilon},\alpha_{s-}^{\varepsilon},z)) - \hat{\Phi}_m(X_{s-}^{\varepsilon},\alpha_{s-}^{\varepsilon})\right] \tilde{N}(ds,dz).
\end{align*}
Using estimate \eref{PhiApp2} and similar argument in the proof of \eref{F3.1.1}, we can deduce that
\begin{align*}
\mathbb{E}\left|\int_{t_0}^t Q(X_{s}^{\varepsilon})\hat{\Phi}_m(X_{s}^{\varepsilon},\cdot)(\alpha^{\varepsilon}_{s})ds\right|^2 \leq C_{T}(1+|x|^{2k+2})\left(m^2\varepsilon^2 + \varepsilon\right).
\end{align*}
This and \eref{F4.19} yield
\begin{align*}
\mathbb{E}\left|\int_{t_0}^t H(X_{s}^{\varepsilon_k},\alpha^{\varepsilon_k}_{s})ds\right|^2 \leq C_{T}(1+|x|^{2k+2})\left(m^{-2} + m^2\varepsilon^2 + \varepsilon\right).
\end{align*}
Subsequently, taking $m = [1/\varepsilon]$, here $[1/\varepsilon]$ is the integer part of $1/\varepsilon$, then it is easy to verify that \eref{ConAV} holds.
\end{proof}

\vspace{0.1cm}
Now,  we  give the proof of our Theorem 2.3.

\textbf{Proof of Theorem \ref{main result 1}}: The proof will be divided  two steps:

\textbf{Step 1:} According to Proposition \ref{pro4.6}, it suffices to prove that for any sequence $\varepsilon_k\to 0$, there exists a subsequence (still denoted by $\varepsilon_k$) such that $X^{\varepsilon_k}$ converges weakly to $X$ in $C([0,T];\mathbb{R}^{n})$ (short by $X^{\varepsilon_k}{\overset{w}\longrightarrow} X$). Without loss of generality, we may assume that $X^{\varepsilon_k}$ converges almost surely to $X$ in $C([0,T];\mathbb{R}^{n})$ by using the Skorohod representation theorem. Next, we will use the martingale problem method to characterize the limit process $X$ as the unique solution of Eq.\eref{bar}.

For given $0 \leq t_0 \leq T$, let $\Psi_{t_0}(\cdot)$ be a bounded continuous function on $C([0,T],\mathbb{R}^n)$ that is measurable with respect to $\sigma\{\varphi_t, \varphi \in C([0,T],\mathbb{R}^n), t \leq t_0\}$. We first prove that for any $U \in C^4_b(\mathbb{R}^n)$, the following assertion holds:
\begin{eqnarray}
\mathbb{E}\left[\left(U(X_{t})-U(X_{t_{0}})-\int^{t}_{t_{0}}\bar{\mathscr{L}}U( X_{s})ds\right)\Psi_{t_{0}}({X})\right]=0,\quad t_0 \leq t \leq T,\label{Mar11}
\end{eqnarray}
where $\bar{\mathscr{L}}$ is the infinitesimal generator of Eq.\eref{bar}, i.e.,
\begin{eqnarray}
\bar{\mathscr{L}}U(x):=\langle\nabla U(x), \bar B(x)\rangle+\frac{1}{2}\text{Tr}\left[\nabla^2U(x)\bar\Sigma(x)\right],\quad x\in\mathbb{R}^n.\label{Generator}
\end{eqnarray}

Using \eref{F3.15}, we have
\begin{align*}
d\Phi(X_{t}^{\varepsilon_k},\alpha_{t}^{\varepsilon_k})=&
\frac{1}{\sqrt{\varepsilon_k}}\partial_x\Phi_K(X_{t}^{\varepsilon_k},\alpha_{t}^{\varepsilon_k})dt+\partial_x\Phi(X_{t}^{\varepsilon_k},\alpha_{t}^{\varepsilon_k})\cdot b(X_{t}^{\varepsilon_k},\alpha_{t}^{\varepsilon_k})dt\\
&+\partial_x\Phi(X_{t}^{\varepsilon_k},\alpha_{t}^{\varepsilon_k})\cdot \sigma(X_{t}^{\varepsilon_k},\alpha_{t}^{\varepsilon_k})dW_t\\
	&+\frac{1}{\varepsilon_k}Q(X_{t}^{\varepsilon_k})\Phi(X_{t}^{\varepsilon_k},\cdot)( \alpha_t^{\varepsilon_k})dt+\frac{1}{\sqrt{\varepsilon_k}}F(X_{t}^{\varepsilon_k},\alpha_t^{\varepsilon_k})dt\\
&+\frac{1}{2}\text{Tr}[\partial_x^{2}\Phi(X_{t}^{\varepsilon_k},\alpha_{t}^{\varepsilon_k})\cdot \left(\sigma\sigma^{*}\right)(X_{t}^{\varepsilon_k},\alpha_{t}^{\varepsilon_k})]dt\\
	&+\int_{[0,\infty)}\left[ \Phi(X_{t-}^{\varepsilon_k},\alpha_{t-}^{\varepsilon_k}+\!g^{\varepsilon}(X_{t-}^{\varepsilon_k},\alpha_{t-}^{\varepsilon_k},z))-\Phi(X_{t-}^{\varepsilon_k},\alpha_{t-}^{\varepsilon_k})\right]\tilde{N}(dt,dz)
\end{align*}
and Ito's formula gives
\begin{align*}
	d\nabla U(X_{t}^{\varepsilon_k})=&\frac{1}{\sqrt{\varepsilon_k}}\nabla^2 U(X_{t}^{\varepsilon_k})\cdot K(X_{t}^{\varepsilon_k},\alpha_{t}^{\varepsilon_k})dt+\nabla^2 U(X_{t}^{\varepsilon_k})\cdot b(X_{t}^{\varepsilon_k},\alpha_{t}^{\varepsilon_k}) dt\\
	&+ \nabla^2 U(X_{t}^{\varepsilon_k})\cdot \sigma(X_{t}^{\varepsilon_k},\alpha_{t}^{\varepsilon_k})d W_t+\frac{1}{2}\text{Tr}[\nabla^3U(X_{t}^{\varepsilon_k})\cdot \left(\sigma\sigma^{*}\right)(X_{t}^{\varepsilon_k}, \alpha_t^{\varepsilon_k})]dt.
\end{align*}
Then, using the integration by parts formula, for any $t \geq t_0$, we obtain
\begin{align*}
	&\big\langle\Phi(X_{t}^{\varepsilon_k},\alpha_{t}^{\varepsilon_k}), \nabla U(X_{t}^{\varepsilon_k})\big\rangle\\
	=&\big\langle\Phi(X_{t_0}^{\varepsilon_k},\alpha^{\varepsilon_k}_{t_0}), \nabla U(X_{t_0}^{\varepsilon_k})\big\rangle+\int^t_{t_0}\big\langle\Phi(X_{s}^{\varepsilon_k},\alpha_{s}^{\varepsilon_k}), d\nabla U(X_{s}^{\varepsilon_k})\big\rangle\\
	&+\int^t_{t_0}\big\langle\nabla U(X_{s}^{\varepsilon_k}), d\Phi(X_{s}^{\varepsilon_k},\alpha_{s}^{\varepsilon_k})\big\rangle+\int^t_{t_0} d[\Phi(X^{\varepsilon_k},\alpha^{\varepsilon_k}), \nabla U(X^{\varepsilon_k})]_s\\
	=&\big\langle \Phi(X_{t_0}^{\varepsilon_k},\alpha^{\varepsilon_k}_{t_0}), \nabla U(X_{t_0}^{\varepsilon_k}) \big\rangle
+\frac{1}{\sqrt{\varepsilon_k}}\int^t_{t_0}\left\langle\Phi(X_{s}^{\varepsilon_k},\alpha_{s}^{\varepsilon_k}),\nabla^2 U(X_{s}^{\varepsilon_k})\cdot K(X_{s}^{\varepsilon_k}, \alpha^{\varepsilon_k}_s) \right\rangle ds\\
	&+\frac{1}{\sqrt{\varepsilon_k}}\int^t_{t_0}\left\langle\nabla U(X_{s}^{\varepsilon_k}),\partial_x \Phi_K(X_{s}^{\varepsilon_k}, \alpha^{\varepsilon_k}_s) +F(X_{s}^{\varepsilon_k},\alpha^{\varepsilon_k}_{s}) \right\rangle ds\\
	&+\frac{1}{\varepsilon_k}\int^{t}_{t_0}\left\langle \nabla U(X_{s}^{\varepsilon_k}),Q(X_{s}^{\varepsilon_k})\Phi(X_{s}^{\varepsilon_k},\cdot)(\alpha^{\varepsilon_k}_{s})\right\rangle ds\\
&+\int^{t}_{t_0}\int_{[0,\infty)}\left\langle \nabla U(X_{s}^{\varepsilon_k}),\left[ \Phi(X_{s-}^{\varepsilon_k},\alpha_{s-}^{\varepsilon_k}+\!g^{\varepsilon}(X_{s-}^{\varepsilon_k},\alpha_{s-}^{\varepsilon_k},z))-\Phi(X_{s-}^{\varepsilon_k},\alpha_{s-}^{\varepsilon_k})\right]\right\rangle\tilde{N}(ds,dz)\\
&+\mathbb{R}^{\varepsilon_k}_{\nabla U,\Phi}(t_0,t),
\end{align*}
where
\begin{align}
&\mathbb{R}^{\varepsilon_k}_{\nabla U,\Phi}(t_0,t)   \nonumber\\
:=&\int^t_{t_0}\big\langle\Phi(X_{s}^{\varepsilon_k},\alpha_{s}^{\varepsilon_k}),\big[\nabla^2 U(X_{s}^{\varepsilon_k})\cdot b(X_{s}^{\varepsilon_k},\alpha_{s}^{\varepsilon_k})+\frac12\text{Tr}[\nabla^3 U(X_{s}^{\varepsilon_k})\cdot (\sigma\sigma^{*})(X_{s}^{\varepsilon_k},\alpha_{s}^{\varepsilon_k})]\big]\big\rangle ds \nonumber\\
&+\int^t_{t_0}\big\langle\nabla U(X_{s}^{\varepsilon_k}),\big[\partial_x\Phi(X_{s}^{\varepsilon_k},\alpha_{s}^{\varepsilon_k})\cdot b(X_{s}^{\varepsilon_k},\alpha_{s}^{\varepsilon_k})+\frac12
\text{Tr}[\partial_x^{2}\Phi(X_{s}^{\varepsilon_k},\alpha_{s}^{\varepsilon_k})\cdot \left(\sigma\sigma^{*}\right)(X_{s}^{\varepsilon_k},\alpha_{s}^{\varepsilon_k})]\big]\big\rangle ds\nonumber\\
&+ \int^t_{t_0}\big\langle\Phi(X_{s}^{\varepsilon_k},\alpha_{s}^{\varepsilon_k}),\nabla^2 U(X_{s}^{\varepsilon_k})\cdot \sigma(X_{s}^{\varepsilon_k},\alpha_{s}^{\varepsilon_k})dW_s\big\rangle \nonumber\\
&+\int^t_{t_0}\big\langle\Phi(X_{s}^{\varepsilon_k},\alpha_{s}^{\varepsilon_k}),\partial_x\Phi(X_{s}^{\varepsilon_k},\alpha_{s}^{\varepsilon_k})\cdot \sigma(X_{s}^{\varepsilon_k},\alpha_{s}^{\varepsilon_k})dW_s\big\rangle\nonumber\\
&+ \int^t_{t_0}\!\text{Tr}\! \left[\left(\partial_x\Phi(X_{s}^{\varepsilon_k},\alpha_{s}^{\varepsilon_k})
\sigma(X_{s}^{\varepsilon_k},\alpha_{s}^{\varepsilon_k})\right)^{\ast}
\left(\nabla^2 U(X_{s}^{\varepsilon_k})
\sigma(X_{s}^{\varepsilon_k},\alpha_{s}^{\varepsilon_k}) \right)\right]ds.\label{3.32}
\end{align}

According to Poisson equation \eref{PE1}, we have
\begin{align}
&\frac{1}{\sqrt{\varepsilon_k}}\int_{t_0}^{t}\left \langle\nabla U(X_{s}^{\varepsilon_k}),K(X^{\varepsilon_k}_s, \alpha^{\varepsilon_k}_s)\right\rangle ds\nonumber\\
=&-\frac{1}{\sqrt{\varepsilon_k}}\int_{t_0}^{t} \left\langle \nabla U(X_{s}^{\varepsilon_k}),Q(X_{s}^{\varepsilon_k})\Phi(X_{s}^{\varepsilon_k},\cdot)(\alpha^{\varepsilon_k}_{s})\right\rangle ds\nonumber\\
=&\sqrt{\varepsilon_k}\left[\mathbb{R}^{\varepsilon_k}_{\nabla U,\Phi}(t_0,t)+\langle\Phi(X_{t_0}^{\varepsilon_k},\alpha^{\varepsilon_k}_{t_0}), \nabla U(X_{t_0}^{\varepsilon_k})\rangle-\langle\Phi(X_{t}^{\varepsilon_k},\alpha_{t}^{\varepsilon_k}), \nabla U(X_{t}^{\varepsilon_k})\rangle\right]\nonumber\\
&+\int_{t_0}^{t}\left\langle\Phi(X_{s}^{\varepsilon_k},\alpha_{s}^{\varepsilon_k}),\nabla^2 U(X_{t}^{\varepsilon_k})\cdot K(X^{\varepsilon_k}_s,\alpha^{\varepsilon_k}_s)\right\rangle ds\nonumber\\
&+\int_{t_0}^{t}\left\langle\nabla U(X_{s}^{\varepsilon_k}),\partial_x\Phi_K(X_{s}^{\varepsilon_k},\alpha_{s}^{\varepsilon_k})+F(X_{s}^{\varepsilon_k},\alpha^{\varepsilon_k}_{s}) \right\rangle ds\nonumber\\
&+\sqrt{\varepsilon_k}\int_{t_0}^{t}\int_{[0,\infty)}\left\langle\nabla U(X_{s}^{\varepsilon_k}),\left[ \Phi(X_{s}^{\varepsilon_k},\alpha_{s-}^{\varepsilon_k}
+\!g^{\varepsilon}(X_{s}^{\varepsilon_k},\alpha_{s-}^{\varepsilon_k},z))-\Phi(X_{s}^{\varepsilon_k},\alpha_{s-}^{\varepsilon_k})\right]\right\rangle\tilde{N}(ds,dz).
\label{F4.11}
\end{align}

Notice that
\begin{align*}
	&\int_{t_0}^{t}\left\langle\Phi(X_{s}^{\varepsilon_k},\alpha_{s}^{\varepsilon_k}),\nabla^2 U(X_{s}^{\varepsilon_k})\cdot K(X^{\varepsilon_k}_s,\alpha^{\varepsilon_k}_s)\right\rangle ds\\
	=&\sum_{i=1}^{n}\sum_{j=1}^{n}\int_{t_0}^{t} \partial _{x_i}\partial_{x_j} U(X_{s}^{\varepsilon_k}) (K\otimes\Phi)_{ij}(X_{s}^{\varepsilon_k},\alpha^{\varepsilon_k}_{s})ds\\
	=&\frac{1}{2}\sum_{i=1}^{n}\sum_{j=1}^{n}\int_{t_0}^{t} \partial _{x_i}\partial_{x_j} U(X_{s}^{\varepsilon_k}) \left[(K\otimes\Phi)+(K\otimes\Phi)^{\ast}\right]_{ij}(X_{s}^{\varepsilon_k},\alpha^{\varepsilon_k}_{s})ds\\
=&\frac{1}{2}\int_{t_0}^{t} \text{Tr}\left[\nabla^2 U(X_{s}^{\varepsilon_k}) \left[(K\otimes\Phi)+(K\otimes\Phi)^{\ast}\right](X_{s}^{\varepsilon_k},\alpha^{\varepsilon_k}_{s})\right]ds.
\end{align*}
Using It\^{o}'s formula, we can obtain
\begin{align*}
U(X_{t}^{\varepsilon_k})
=& U(X_{t_0}^{\varepsilon_k})+\int_{t_0}^{t}\left\langle\nabla U(X_{s}^{\varepsilon_k}),b(X_{s}^{\varepsilon_k},\alpha^{\varepsilon_k}_{s})\right\rangle ds+\int_{t_0}^{t}\left\langle\nabla U(X_{s}^{\varepsilon_k}),\sigma(X_{s}^{\varepsilon_k},\alpha^{\varepsilon_k}_{s})dW_s\right\rangle\nonumber\\
&+\frac{1}{2}\int_{t_0}^{t}
\text{Tr}[\nabla^2U(X_{s}^{\varepsilon_k})\cdot \left(\sigma\sigma^{*}\right)(X_{s}^{\varepsilon_k},\alpha^{\varepsilon_k}_{s})]ds\nonumber\\
&+\frac{1}{\sqrt{\varepsilon_k}}\int_{t_0}^{t}\left\langle\nabla U(X_{s}^{\varepsilon_k}),K(X_{s}^{\varepsilon_k},\alpha^{\varepsilon_k}_{s})\right\rangle ds\\
=& U(X_{t_0}^{\varepsilon_k})+\int_{t_0}^{t}\left\langle\nabla U(X_{s}^{\varepsilon_k}),b(X_{s}^{\varepsilon_k},\alpha^{\varepsilon_k}_{s})\right\rangle ds+\int_{t_0}^{t}\left\langle\nabla U(X_{s}^{\varepsilon_k}),\sigma(X_{s}^{\varepsilon_k},\alpha^{\varepsilon_k}_{s})dW_s\right\rangle\nonumber\\
&+\frac{1}{2}\int_{t_0}^{t}
\text{Tr}[\nabla^2 U(X_{s}^{\varepsilon_k})\cdot \left(\sigma\sigma^{*}\right)(X_{s}^{\varepsilon_k},\alpha^{\varepsilon_k}_{s})]ds\nonumber\\
&+\frac{1}{2}\int_{t_0}^{t} \text{Tr}\left[\nabla^2 U(X_{s}^{\varepsilon_k}) \left[(K\otimes\Phi)+(K\otimes\Phi)^{\ast}\right](X_{s}^{\varepsilon_k},\alpha^{\varepsilon_k}_{s})\right]ds\nonumber\\
&+\int_{t_0}^{t}\left\langle\nabla U(X_{s}^{\varepsilon_k}),\partial_x\Phi_K(X_{s}^{\varepsilon_k},\alpha_{s}^{\varepsilon_k})+F(X_{s}^{\varepsilon_k},\alpha^{\varepsilon_k}_{s}) \right\rangle ds\nonumber\\
&+\sqrt{\varepsilon_k}\int_{t_0}^{t}\int_{[0,\infty)}\left\langle\nabla U(X_{s}^{\varepsilon_k}),\left[ \Phi(X_{s-}^{\varepsilon_k},\alpha_{s-}^{\varepsilon_k}+\!g^{\varepsilon}(X_{s-}^{\varepsilon_k},\alpha_{s-}^{\varepsilon_k},z))
-\Phi(X_{s-}^{\varepsilon_k},\alpha_{s-}^{\varepsilon_k})\right]\right\rangle\tilde{N}(ds,dz)\nonumber\\
&+\sqrt{\varepsilon_k}\left[\mathbb{R}^{\varepsilon_k}_{\nabla U,\Phi}(t_0,t)\!+\!\left\langle\Phi(X_{t_0}^{\varepsilon_k},\alpha^{\varepsilon_k}_{t_0}), \nabla U(X_{t_0}^{\varepsilon_k})\right\rangle\!-\!\left\langle\Phi(X_{t}^{\varepsilon_k},\alpha_{t}^{\varepsilon_k}), \nabla U(X_{t}^{\varepsilon_k})\right\rangle\right].\label{3.31}
\end{align*}

According to \eref{Phi} and $U\in C^4_b(\mathbb{R}^{n})$, it is easy to see that
\begin{eqnarray}
&&\lim_{k\to\infty}\sqrt{\varepsilon_k}\mathbb{E}\big\{\big[\mathbb{R}^{\varepsilon_k}_{\nabla U,\Phi}(t,t_0)+\left\langle\Phi(X_{t_0}^{\varepsilon_k},\alpha^{\varepsilon_k}_{t_0}), \nabla U(X_{t_0}^{\varepsilon_k})\right\rangle\nonumber\\
&&\quad\quad\quad\quad\quad-\left\langle\Phi(X_{t}^{\varepsilon_k},\alpha_{t}^{\varepsilon_k}), \nabla U(X_{t}^{\varepsilon_k})\right\rangle\big]\Psi_{t_0}(X^{\varepsilon_k})\big\}=0\label{3.35}
\end{eqnarray}
and
\begin{align*}
&\mathbb{E}\left\{\left[\int^{t}_{t_0}\int_{[0,\infty)}\left\langle\nabla U(X_{s}^{\varepsilon_k}),\left[ \Phi(X_{s-}^{\varepsilon_k},\alpha_{s-}^{\varepsilon_k}
+\!g^{\varepsilon}(X_{s-}^{\varepsilon_k},\alpha_{s-}^{\varepsilon_k},z))-\Phi(X_{s-}^{\varepsilon_k},\alpha_{s-}^{\varepsilon_k})\right]\right\rangle\tilde{N}(ds,dz)\right.\right.\\
&\quad\quad+\left.\left.\int^t_{t_0}\left\langle\nabla U(X_{s}^{\varepsilon_k}),\sigma(X_{s}^{\varepsilon_k},\alpha_{s}^{\varepsilon_k})dW_s\right\rangle\right]\Psi_{t_0}(X^{\varepsilon_k})\right\}=0.
\end{align*}

Therefore, \eref{Mar11} holds provided that
\begin{eqnarray}
\lim_{k\rightarrow \infty}\mathbb{E}\left[\Gamma^{\varepsilon_k}\Psi_{t_0}(X^{\varepsilon_k})\right]=\mathbb{E} \left[\Gamma\Psi_{t_0}( X)\right],\label{Step2}
\end{eqnarray}
where
\begin{align*}
    \Gamma^{\varepsilon_k}:=&U(X_{t}^{\varepsilon_k})-U(X_{t_0} ^{\varepsilon_k})-\int_{t_0}^t \langle \nabla U(X_{s}^{\varepsilon_k}), \left(b+\partial_x\Phi_{K}+F\right)(X_{s}^{\varepsilon_k},\alpha^{\varepsilon_k}_{s}) \rangle ds\nonumber\\
    &-\frac{1}{2}\int_{t_0}^t \text{Tr}\left[\nabla^2 U(X_{s}^{\varepsilon_k}) \left[(\sigma\sigma^{*})+(K\otimes\Phi)+(K\otimes\Phi)^{\ast}\right](X_{s}^{\varepsilon_k},\alpha^{\varepsilon_k}_{s})\right]ds
\end{align*}
and
\begin{eqnarray*}
    \Gamma:=U(X_{t})-U(X_{t_0})-\int_{t_0}^t\mathscr{\bar L}U(X_{s})ds.
\end{eqnarray*}

\textbf{Step 2}: We will prove \eref{Step2} in this step. Since $X^{\varepsilon_k}$ converges almost surely to $X$ in $C([0,T];\mathbb{R}^{n})$, it is easy to see that
\begin{align*}
\lim_{k\rightarrow \infty}\mathbb{E}\left[\left(U(X_{t}^{\varepsilon_k})-U(X_{t_0} ^{\varepsilon_k})\right)\Psi_{t_0}(X^{\varepsilon_k})\right]=\mathbb{E} \left[\left(U(X_{t})-U(X_{t_0})\right)\Psi_{t_0}(X)\right],
\end{align*}
\begin{align*}
    \lim_{k\to \infty} \mathbb{E}\left\{\left[\int_{t_0}^t\langle\nabla U(X^{\varepsilon_k}_s),\bar{B}(X^{\varepsilon_k}_s)\rangle ds\right]\Psi_{t_0}(X^{\varepsilon_k})\right\}
    =\mathbb{E}\left\{\left[\int_{t_0}^t\langle\nabla U(X_s),\bar{B}(X_s)\rangle ds\right]\Psi_{t_0}(X)\right\}
\end{align*}
and
\begin{align*}
    &\lim_{k\to \infty} \mathbb{E}\left\{\left[\int_{t_0}^t\text{Tr}\left[\nabla^2 U(X_{s}^{\varepsilon_k}) \bar\Sigma(X_{s}^{\varepsilon_k})\right]ds\right]\Psi_{t_0}(X^{\varepsilon_k})\right\}\\
    =&\mathbb{E}\left\{\left[\int_{t_0}^t\text{Tr}\left[\nabla^2 U(X_{s}) \bar\Sigma(X_{s})\right]ds\right]\Psi_{t_0}(X)\right\}.
\end{align*}
Therefore,  in order to prove \eref{Step2}, it suffices to show that the followings hold:
\begin{eqnarray}
\lim_{k\to \infty} \mathbb{E}\left|\int^t_{t_0}\langle\nabla U(X^{\varepsilon_k}_s),B(X^{\varepsilon_k}_s, \alpha^{\varepsilon_k}_s)-\bar B(X^{\varepsilon_k}_s)\rangle ds\right|=0.\label{3.39}
\end{eqnarray}
and
\begin{eqnarray}
\lim_{k\rightarrow \infty} \mathbb{E}\left|\int^t_{t_0}\text{Tr}\left[\nabla^2 U(X_{s}^{\varepsilon_k}) [\Sigma(X_{s}^{\varepsilon_k},\alpha^{\varepsilon_k}_{s})-\bar\Sigma(X_{s}^{\varepsilon_k})]\right]ds\right|=0.\label{3.37}
\end{eqnarray}

Denote
\begin{align*}
&\Psi_1(x,i)=\text{Tr}\left[\nabla^2 U(x) \left[\Sigma(x,i)-\bar\Sigma(x)\right]\right],\\
&\Psi_2(x,i)=\langle\nabla U(x),B(x, i)-B(x)\rangle.
\end{align*}
It is easy to see that $\Psi_j$ ($j=1,2$) satisfies the "centering condition" \eref{CenCon}, and has the followings
\begin{align*}
\|\Psi_j(x,\cdot)-\Psi_j(y,\cdot)\|_{\infty}\leq C(1+|x|+|y|)|x-y|,\quad \|\Psi_j(x,\cdot)\|_{\infty}\leq C(1+|x|).
\end{align*}
Proposition \ref{Pro4.2} implies that \eref{3.39} and \eref{3.37} hold.

Finally,  since the uniqueness of the martingale problem solution for the operator $\bar{\mathscr{L}}$ is equivalent to the uniqueness of the weak solution of Eq.\eref{bar},  uniqueness of strong solution of Eq.\eref{bar} implies that the accumulation point $X$ is the unique solution of Eq.\eref{bar} in the weak sense.   \hfill\fbox

\section{Weak Convergence Order of $X^{\varepsilon}_t$ in $\mathbb{R}^n$}
\vspace{0.1cm}
In this section,  we will study the convergence order of $|\mathbb{E}\phi(X^{\varepsilon}_t) - \mathbb{E}\phi(\bar{X}_t)|$ ($t\le T$) for suitable test functions $\phi$.

Considering the Kolmogorov equation
\begin{equation}\label{KE1}
\begin{cases}
\displaystyle
\partial_t u(t,x) = \bar{\mathscr{L}} u(t,x), & t \in [0, T], \\
u(0, x) = \phi(x), &
\end{cases}
\end{equation}
where $\phi \in C^{4}_p(\mathbb{R}^n)$ and $\bar{\mathscr{L}}$ is defined in \eref{Generator} ,which is the generator of Eq.\eref{bar}. Clearly, Eq.\eref{KE1} has a unique solution $u$ given by
$$
u(t,x) = \mathbb{E}\phi(\bar{X}^x_t), \quad t \geq 0.
$$

Using the estimate \eref{RE2.4},  we can obtain the regularity estimates for the solution $u$ of Eq.\eref{KE1}, since the proof is almost the same as that of \cite[Lemma 5.1]{SSWX2024}, we omit it.

\begin{lemma} \label{Lemma_3.3.2}
For  $T > 0$ and $\phi \in C^{4}_p(\mathbb{R}^n)$, there exist $C_{T,\phi} > 0$ and $k > 0$ such that for  $x \in \mathbb{R}^n$,
$$
\sup_{0 \leq t \leq T} \sum_{j=0}^{4} \|\partial^j_{x} u(t,x)\| \leq C_{T,\phi}(1 + |x|^{k}),
$$
$$
\sup_{0 \leq t \leq T} \sum_{j=1}^{2} \|\partial_t(\partial^j_x u(t,x))\| \leq C_{T,\phi} (1 + |x|^{k}).
$$
\end{lemma}

\vspace{0.1cm}
Now, we can prove our second main result.

\textbf{Proof of Theorem \ref{main result 2}}:
Fix $t \in (0,T]$, and let $\tilde{u}^t(s,x) := u(t-s,x)$ for $0 \leq s \leq t$.  It\^o's formula implies that
\begin{align*}
\tilde{u}^t(t, X^{\varepsilon}_t) &= \tilde{u}^t(0,x) + \int^t_0 \partial_s \tilde{u}^t(s, X^{\varepsilon}_s) \, ds + \frac{1}{\sqrt{\varepsilon}} \int^t_0 \langle K(X^{\varepsilon}_s,\alpha^{\varepsilon}_s),\partial_x\tilde{u}^t(s,X^{\varepsilon}_s)\rangle \, ds \\
&\quad + \int^t_0 \mathscr{L}_{1}(\alpha^{\varepsilon}_s)\tilde{u}^t(s, X^{\varepsilon}_s) \, ds + \int^t_0 \langle\partial_x\tilde{u}^t(s, X^{\varepsilon}_s), \sigma(X^{\varepsilon}_s,\alpha^{\varepsilon}_s) \, dW_s\rangle,
\end{align*}
where
\begin{align*}
\mathscr{L}_{1}(\alpha)\phi(x) &:= \langle b(x,\alpha), \nabla \phi(x)\rangle + \frac{1}{2}\text{Tr}\left[\nabla^2\phi(x)\left(\sigma\sigma^{*}\right)(x,\alpha)\right], \quad \phi \in C^2(\mathbb{R}^n).
\end{align*}

Since $\tilde{u}^t(t,X^{\varepsilon}_t)=\phi(X^{\varepsilon}_t)$,  $\tilde{u}^t(0, x)=\mathbb{E}\phi(\bar{X}^{x}_t)$ and $\partial_s \tilde{u}^t(s, X^{\varepsilon}_s)=-\bar{\mathscr{L}} \tilde{u}^t(s,X^{\varepsilon}_s)$,  we have
\begin{align}
\left|\mathbb{E}\phi(X^{\varepsilon}_{t}) - \mathbb{E}\phi(\bar{X}_{t})\right|=& \Big|\mathbb{E}\int^t_0 -\bar{\mathscr{L}} \tilde{u}^t(s, X^{\varepsilon}_s) \, ds + \mathbb{E}\int^t_0 \mathscr{L}_{1}(\alpha^{\varepsilon}_s)\tilde{u}^t(s,X^{\varepsilon}_s) \, ds\nonumber \\
&\quad\quad +\frac{1}{\sqrt{\varepsilon}} \mathbb{E}\int^t_0 \langle K(X^{\varepsilon}_s,\alpha^{\varepsilon}_s), \partial_x \tilde{u}^t(s,X_{s}^{\varepsilon})\rangle \, ds\Big|. \label{3.41}
\end{align}

By a similar argument in \eref{F4.11},  one has
\begin{align*}
&\frac{1}{\sqrt{\varepsilon}} \int^t_{0} \langle K(X^{\varepsilon}_s,\alpha^{\varepsilon}_s), \partial_x \tilde{u}^t(s,X_{s}^{\varepsilon})\rangle \, ds \nonumber \\
&= -\frac{1}{\sqrt{\varepsilon}} \int^t_{0} \left\langle Q(X_{s}^{\varepsilon})\Phi(X_{s}^{\varepsilon},\cdot)(\alpha^{\varepsilon}_{s}), \partial_x \tilde{u}^t(s,X_{s}^{\varepsilon})\right\rangle \, ds \nonumber \\
&= \sqrt{\varepsilon} \left[ \mathbb{R}^{\varepsilon}_{\partial_x\tilde{u},\Phi}(0,t) + \langle\Phi(x,\alpha), \partial_x \tilde{u}^t(0,x)\rangle - \langle\Phi(X_{t}^{\varepsilon},\alpha_{t}^{\varepsilon}), \partial_x \tilde{u}^t(t,X_{t}^{\varepsilon})\rangle \right] \nonumber \\
&\quad + \int^t_{0} \left\langle \Phi(X_{s}^{\varepsilon},\alpha_{s}^{\varepsilon}), \partial^2_x \tilde u^t(s,X_{s}^{\varepsilon}) \cdot K(X^{\varepsilon}_s,\alpha^{\varepsilon}_s) \right\rangle \, ds \nonumber \\
&\quad + \int^t_{0} \left\langle \partial_x \tilde{u}^t(s,X_{s}^{\varepsilon}), \partial_x\Phi_K(X_{s}^{\varepsilon},\alpha_{s}^{\varepsilon}) + F(X_{s}^{\varepsilon},\alpha^{\varepsilon}_{s}) \right\rangle \, ds \nonumber \\
&\quad + \sqrt{\varepsilon} \int^{t}_{0} \int_{[0,\infty)} \left\langle \partial_x \tilde{u}^t(s,X_{s}^{\varepsilon}), \left[ \Phi(X_{s-}^{\varepsilon},\alpha_{s-}^{\varepsilon} + g^{\varepsilon}(X_{s-}^{\varepsilon},\alpha_{s-}^{\varepsilon},z)) - \Phi(X_{s-}^{\varepsilon},\alpha_{s-}^{\varepsilon}) \right] \right\rangle\tilde{N}(ds,dz),
\end{align*}
where $\mathbb{R}^{\varepsilon}_{\partial_x \tilde{u}^t,\Phi}(0,t)$ is defined in \eref{3.32} with $\varepsilon_k$ and $\nabla U$ are replaced by $\varepsilon$ and $\partial_x \tilde{u}^t$ respectively.	

Therefore, we have
\begin{align}
&\left|\mathbb{E}\phi(X^{\varepsilon}_{t})-\mathbb{E}\phi(\bar{X}_{t})\right|\nonumber\\
\leq&\Big|\mathbb{E}\!\int^t_0 \!\langle\partial_x \tilde{u}^t(s, X^{\varepsilon}_s ), B(X^{\varepsilon}_s,\alpha^{\varepsilon}_s)-\bar B(X^{\varepsilon}_s)\rangle\nonumber\\
&\quad\quad\quad +\frac{1}{2}\text{Tr}\left[\partial_x^{2}\tilde{u}^t(s, X^{\varepsilon}_s)\left(\Sigma(X^{\varepsilon}_s,\alpha^{\varepsilon}_s)-\bar\Sigma(X^{\varepsilon}_s)\right)\right]  ds\Big|\nonumber\\
&+\sqrt{\varepsilon}\Big|\mathbb{E}\big [\mathbb{R}^{\varepsilon}_{\partial_x\tilde{u}^t,\Phi}(0,t)+ \langle\Phi(x,\alpha), \partial_x \tilde{u}^t(0,x)\rangle-\langle\Phi(X_{t}^{\varepsilon},\alpha_{t}^{\varepsilon}), \partial_x \tilde{u}^t(t, \bar{X}_t)\rangle\big]\Big|.\label{F5.3}
\end{align}
 \eref{Phi} and Lemma \ref{Lemma_3.3.2} imply that
\begin{align}
&\sqrt{\varepsilon}\Big[\mathbb{R}^{\varepsilon}_{\partial_x\tilde{u},\Phi}(0,t)+ \langle\Phi(x,\alpha), \partial_x \tilde{u}^t(0,x)\rangle-\langle\Phi(X_{t}^{\varepsilon},\alpha_{t}^{\varepsilon}), \partial_x \tilde{u}^t(t, \bar{X}_t)\rangle\Big]\nonumber\\
\leq&C_{T}(1+|x|^{k})\sqrt{\varepsilon}.\label{F5.4}
\end{align}

Denote
\begin{align*}
\Psi_3(s,x,i)=\langle B(x, i)-B(x),  \partial_x \tilde{u}^t(s,x)\rangle
+\frac{1}{2}\text{Tr}\left[\partial_x^{2}\tilde{u}^t(s, x)\left(\Sigma(x,i)-\bar\Sigma(x)\right)\right].
\end{align*}
We have that, for $s,s_1,s_2\in [0,T]$ and $x,y\in\mathbb{R}^n$,
$$\|\Psi_3(s_1,x,\cdot)-\Psi_3(s_2,y,\cdot)\|_{\infty}\leq C_T(1+|x|^k+|y|^k)\left(|x-y|+|s_1-s_2|\right),$$
$$\sup_{s\in [0,T]}\|\Psi_3(s,x,\cdot)\|_{\infty}\leq C(1+|x|^k).$$
Similarly to the proof of \eref{ConAV}, we can easily obtain:
\begin{equation}
\mathbb{E}\left|\int^t_{t_0}\Psi_3(s,X_{s}^{\varepsilon},\alpha^{\varepsilon}_{s})ds\right|\leq C(1+|x|^{k+1})\varepsilon^{1/2}.\label{ConAV2}
\end{equation}

Finally, according to \eref{F5.3}-\eref{ConAV2}, we have
\begin{equation*}
\begin{aligned}
\sup_{0\leq t\leq T}\left|\mathbb{E}\phi(X^{\varepsilon}_{t})-\mathbb{E}\phi(\bar{X}_{t})\right|
\leq C_{T}(1+|x|^{k+1})\varepsilon^{\frac{1}{2}}.
\end{aligned}
\end{equation*}
The proof is complete.  \hfill\fbox

\vspace{0.3cm}
\textbf{Acknowledgment}. The research of Xiaobin Sun is supported by the NSF of China (Nos.
12271219, 12090010 and 12090011). The research of Yingchao Xie is supported by the NSF of China  (No 12471139) and the Priority Academic Program Development of Jiangsu Higher Education Institutions.

\end{document}